\documentclass[leqno]{article}

\pdfoutput=1

\usepackage[english]{babel}

\usepackage[letterpaper,top=2cm,bottom=2cm,left=3cm,right=3cm,
marginparwidth=1.75cm]{geometry}

\usepackage{amsmath,amsthm,amssymb,tikz,tikz-cd,ytableau,relsize}
\usepackage{stmaryrd,graphicx,xcolor}
\usepackage[colorlinks=true, allcolors=blue]{hyperref}
\usetikzlibrary{calc,tikzmark}

\newcommand{\bigast}{\mathbin{\scalebox{1.5}{\ensuremath{\ast}}}}
\newcommand{\vf}{\vfill\end{document}}

\newcommand{\centersmallbullet}{{}_{{}^{{}^{
\scriptscriptstyle{\bullet\!}}}}}

\newtheorem{Theorem}[equation]{Theorem}
\newtheorem{Corollary}[equation]{Corollary}
\newtheorem{Property}[equation]{Property}
\newtheorem{Lemma}[equation]{Lemma}
\newtheorem{Fact}[equation]{Fact}

\theoremstyle{definition}

\newtheorem{Definition}[equation]{Definition}
\newtheorem{Notation}[equation]{Notation}
\newtheorem{Example}[equation]{Example}
\newtheorem{Remark}[equation]{Remark}

\numberwithin{equation}{section}

\title{Global sections of the positively twisted 
Green-Griffiths bundles}

\author{Victor Chen, Jo\"el Merker}

\begin{document}

\maketitle

\begin{abstract}
With various jet orders $k$ and weights $n$,
let $E_{k,n}^{\rm GG}$ be the Green-Griffiths bundles over
the projective space $\mathbb{P}^N (\mathbb{C})$.
Denote by $\mathcal{O} (d)$ the tautological line bundle 
over $\mathbb{P}^N (\mathbb{C})$.

Although only {\em negative} twists are of interest
for applications to complex hyperbolicity
(above general type projective submanifolds $Y \subset 
\mathbb{P}^N (\mathbb{C})$),
it is known that the {\em positive} twists 
$E_{k,n}^{\rm GG} \otimes \mathcal{O} (d)$
enjoy nontrivial global sections.

In this article, we establish that for every $d \geqslant 1$
and for every jet order $k \geqslant d-1$:
\[
\dim\,
H^0
\bigg(
\mathbb{P}^N,\,\,
\bigoplus_{n=1}^{\infty}
E_{k, n}^{\text{GG}} 
\otimes 
\mathcal{O}(d)
\bigg)
=
(N+1)^d.
\]

This theorem is actually a corollary of a recent work of Etesse,
devoted to a proof,
from the point of view of {\sl differentially homogeneous}
polynomials, of the so-called Schmidt-Kolchin-Reinhart
conjecture, by means of (advanced) Representation Theory. As Etesse
discovered a (simple) tight link with the Green-Griffiths bundles,
both statements are in fact equivalent.

Our objective is to set up an alternative proof of the above precise
dimension estimate, from the
Green-Griffiths point of view (only). 
More precisely, we find an explicit description
of all concerned global sections. 
Our arguments are elementary, 
and use only determinants,
linear algebra, monomial orderings.

One old hope is to discover some explicit formulas
for global sections of
{\em negatively twisted} Green-Griffiths bundles 
over projective general type submanifolds 
$Y \subset \mathbb{P}^N (\mathbb{C})$,
a problem still open.

\end{abstract}

\tableofcontents

%%%%%%%%%%%%%%%%%%%%%%%%%%%%%%%%%%%%%%%%%%%%%%%%%%%%%%%%%%%%%%%%%%%%%%
\section{Introduction}
%%%%%%%%%%%%%%%%%%%%%%%%%%%%%%%%%%%%%%%%%%%%%%%%%%%%%%%%%%%%%%%%%%%%%%

It is known that homogeneous polynomials in the variables $X_0, X_1,
\dots, X_N$ can be dehomogenized to represent holomorphic sections of
the tautological line bundle $\mathcal{O}(d)$ over the projective
space $\mathbb{P}^N(\mathbb{C})$. The homogeneity condition
corresponds to the changes of charts conditions. The converse is also
true. Every global section is in fact affinely polynomial, and
corresponds to some homogeneous polynomial.

A $k$-jet is defined as the $N \times k$-tuple of the successive
derivatives of some holomorphic curve:
\[
\gamma: \mathbb{C} \longrightarrow \mathbb{P}^N.
\]
More precisely, if:
\[
\begin{array}{ccccl}
\phi & : & U & \longrightarrow & \quad \quad \mathbb{C}^N
\\
& & x & \longmapsto & 
\big(
\phi_1(x), \dots, \phi_N(x)
\big)
\\
\end{array}
\]
is a chart defined in some open $U \subset 
\mathbb{P}^N (\mathbb{C})$, then a $k$-jet is a tuple:
\[
\big(x_1', \dots, x_N', x_1'', \dots, x_N'', \dots\big)
=
\big(
(\phi_1 \circ \gamma)'(0),
\dots,
(\phi_N  \circ \gamma)'(0),
(\phi_1 \circ \gamma)''(0),
\dots,
(\phi_N \circ \gamma)''(0),
\dots
\big).
\]

It can be verified that, as soon as second derivatives are involved,
the bundle of $k$-jets is not a vector bundle. To recover linearity,
one can consider the polynomials in the variables $x_1', \dots,
x_N^{(k)}$ whose monomials all share the same {\sl weight}. Here, the
weight of a monomial $x_{i_1}^{(l_1)}x_{i_2}^{(l_2)} \cdots
x_{i_t}^{(l_t)}$ is defined as the sum $l_1 + l_2 + \cdots + l_t$
of the orders of derivatives.

Consider the bundle whose fibers are given by these polynomials, and
denote their shared weight by $n$. This bundle is the {\sl
Green-Giffiths bundle} of order $k$ and weight $n$, denoted $E_{k,
n}^{\text{GG}}$, and it is a vector bundle
(\cite{Demailly-2015, Merker-2015}).
Denote by $\mathcal{O} (d)$ the tautological line bundle 
over $\mathbb{P}^N (\mathbb{C})$.

In this paper, we are interested in the 
global
sections of the positive twists
of the Green-Griffiths bundle, namely in:
\[
H^0
\Big(
\mathbb{P}^N(\mathbb{C}),\,
E_{k, n}^{\text{GG}} \otimes \mathcal{O}(d)
\Big),
\]
for some integer $d \geqslant 1$. Furthermore, we will only consider
these bundles over the whole projective space $\mathbb{P}^N$.

Our main theorem gives the dimension of the space of holomorphic
sections.

\begin{Theorem}
\label{Thm-N-1-d}
For every integer $d \geqslant 1$ 
and for every jet order $k \geqslant
d-1$, one has:
\[
\dim
H^0
\bigg(
\mathbb{P}^N,\,\,
\bigoplus_{n=1}^{\infty}
E_{k, n}^{\text{GG}} 
\otimes 
\mathcal{O}(d)
\bigg)
=
(N+1)^d.
\]
\end{Theorem}

This theorem is actually a corollary of a recent work of Etesse.
In~\cite{Etesse-2023},
Etesse shows that holomorphic sections of
$E_{k, n}^{\text{GG}}
\otimes \mathcal{O}(d)$ are in
one-to-one correspondence
with {\sl differentially homogeneous polynomials} of degree $d$.
These
polynomials $P$
constitute a subspace $V_d^{\text{diff}} \subset V_d$ 
of the space $V_d$ of all
polynomials
which are homogeneous of degree $d$ with respect to all 
the jet variables $X_i^{(l)}$. By definition, 
a polynomial $P = P \big( \big\{ X_i^{(l)} \big\} \big)$
in $V_d$
belongs to $V_d^{\rm diff}$ if and only if, 
for any polynomial $Q \in \mathbb{C} [T]$ with $T$ being an extra
`time' variable, one has:
\[
P
\Big(
\big\{
\big(
Q\,X_i
\big)^{(l)}
\big\}
\Big)
\,=\,
Q^d\,
P
\big(
\big\{
X_i^{(l)}
\big\}
\big),
\]
where, by Leibniz' rule:
\[
\big(
Q\,X_i
\big)^{(l)}
=
\sum_{s=0}^l\,
{\textstyle{\binom{l}{s}}}\,
Q^{(s)}\,
X_i^{(l-s)},
\]
so that {\em no} derivatives of $Q$ remain\,\,---\,\,only
$Q^d$ appears. The correspondence is as follows.

\begin{Theorem}[Etesse]
For every $k \geqslant d-1$, the map:
\[
\begin{array}{ccccl}
V_d^{\text{\rm diff}} & \longrightarrow & \quad \quad 
H^0\Big(
\mathbb{P}^N,\,\, 
\bigoplus_{n=1}^{\infty} E_{k, n}^{\text{GG}} 
\otimes \mathcal{O}(d)
\Big)
\\
P & \longmapsto & 
\Big(\frac{P}{X_0^d}, \frac{P}{X_1^d}, \dots, 
\frac{P}{X_N^d}
\Big)
\\
\end{array}
\]
is an isomorphism.\qed
\end{Theorem}

In the same paper~\cite{Etesse-2023}, 
Etesse was able to prove the following result on
the dimension of the space of differentially homogeneous polynomials,
formerly known as the {\sl Schmidt-Kolchin-Reinhart Conjecture}.

\begin{Theorem}[Etesse-Reinhart-Schmidt-Kolchin]
For every integer $d$, the space $V_d^{\rm diff}$ 
of differentially homogeneous polynomials 
of degree $d$ has dimension $(N+1)^d$.\qed
\end{Theorem}

This conjecture was stated for the first time by Schmidt
in~\cite{Schmidt-1979}. Schmidt found $(N+1)^d$ as a lower
bound, and he surmised `{\sl Perhaps, equality is true here}'.

In~\cite{Reinhart-1999}, Reinhart proved the conjecture in
dimension $N = 1$, and he
also provided a construction for a free family of $(N+1)^d$ elements in
$V_d^{\text{diff}}$. Such a family was
obtained by taking determinants of
certain $d \times d$ matrices.

In~\cite{Etesse-2023}, Etesse proved the Schmidt-Kolchin-Reinhart 
conjecture in
the general case, {\em i.e.} in any dimension $N \geqslant 1$. 
His idea is to squeeze the space of differentially
homogeneous polynomials $V_d^{\text{diff}, (k)}$ (with the variables
having formal derivatives all $\leqslant k$):
\[
\overline{V_d^{\text{diff}, (k)}}
\subset
V_d^{\text{diff}, (k)}
\subset
V_d^{(k)},
\]
where $\overline{V_d^{\text{diff}, (k)}}$ is the space generated by
the free family of Reinhart
determinants, and where $V_d^{(k)}$ is the space of all
homogeneous polynomials in the $X_i^{(l)}$
with $l \leqslant k$. 

These spaces can be studied as
representations of the linear group $GL_{N + 1}(\mathbb{C})$, and they
admit the following decompositions as sums of irreductible
subrepresentations:
\[
\bigoplus_{\lambda \vdash d} 
S^{\lambda}(\mathbb{C}^N)^{\oplus f_{\lambda}}
\subset
\bigoplus_{\lambda \vdash d} 
S^{\lambda}(\mathbb{C}^N)^{\oplus \xi_{\lambda}}
\subset
\bigoplus_{\lambda \vdash d} 
S^{\lambda}(\mathbb{C}^N)^{\oplus d_{\lambda}(N + 1)},
\]
where $f_{\lambda}$, $\xi_\lambda$, 
$d_{\lambda} (N+1)$ are combinatorial integers
related to the number of fillings of the {\sl Young diagram} of shape
a partition $\lambda$ of $d$.

By introducing an intermediate algebraic problem, Etesse 
establishes 
that the dimension $\xi_{\lambda}$ of the space generated by
$\lambda$-highest weight vectors must be equal to $f_{\lambda}$,
and he 
concludes that the space $V_d^{\text{diff}, (k)}$ coincides with
the space $\overline{V_d^{\text{diff}, (k)}}$ generated by the
determinants.

\medskip

The main goal of this paper is to describe 
explicitly
the sections of the twisted
Green-Griffiths bundle $E_{k, n}^{\text{GG}} \otimes \mathcal{O}(d)$,
and to establish,
independently and differently,
that its dimension is equal to $(N + 1)^d$.
Over the affine chart $U_0 = \{ X_0 \neq 0\}$, 
such a section can be written as a polynomial in
the space:
\[
E 
= 
\mathbb{C}
\big[
x_{01}, \dots,  x_{0N}, x_{01}', \dots, x_{0N}', \dots
\big].
\]
We need to find the polynomials in $E$ such that, after the
changes of charts from $U_0$
to $U_i = \{X_i \neq 0\}$, namely:
\[
\aligned
x_{ij} 
&=
\frac{x_{0j}}{x_{0i}} 
\ \ \ \ \
\text{if} \ i \neq j,
\\
x_{ii}
&=
\frac{1}{x_{0i}},
\endaligned
\]
the obtained rational function has only poles of order
$\leqslant d$, hence compensated by the
positive twist $ (\centersmallbullet) \otimes \mathcal{O}(d)$.

To do that, we start by introducing a 
better
basis of the space $E$. 
We found this basis
by dehomogenizing the determinants of Reinhart and
Etesse. In affine coordinates $x_{01}, \dots, x_{0N}$ 
on $\mathbb{C}^N \subset \mathbb{P}^N$,
the elements of this basis are certain determinants,
for instance with $N = 3$:
\[
\footnotesize
\aligned
\def\arraystretch{1.75}
\begin{vmatrix}
x_{03}^{'''} & 
x_{03}^{(4)} & 
x_{02}^{''} & 
x_{02}^{(5)} &
x_{01}^{(4)} & 
x_{01}^{(7)} & 
x_{01}^{(9)}
\\
\binom{3}{1}x_{03}^{''} & 
\binom{4}{1}x_{03}^{'''} & 
\binom{2}{1}x_{02}^{'} & 
\binom{5}{1}x_{02}^{(4)} &
\binom{4}{1}x_{01}^{'''} & 
\binom{7}{1}x_{01}^{(6)} & 
\binom{9}{1}x_{01}^{(8)}
\\
\binom{3}{2}x_{03}^{'} & 
\binom{4}{2}x_{03}^{''} & 
\binom{2}{2}x_{02} & 
\binom{5}{2}x_{02}^{'''} &
\binom{4}{2}x_{01}^{''} & 
\binom{7}{2}x_{01}^{(5)} & 
\binom{9}{2}x_{01}^{(7)}
\\
\binom{3}{3}x_{03} & 
\binom{4}{3}x_{03}^{'} & 
0 & 
\binom{5}{3}x_{02}^{''} &
\binom{4}{3}x_{01}^{'} & 
\binom{7}{3}x_{01}^{(4)} & 
\binom{9}{3}x_{01}^{(6)}
\\
0 & 
\binom{4}{4}x_{03} & 
0 & 
\binom{5}{4}x_{02}^{'} &
\binom{4}{4}x_{01} & 
\binom{7}{4}x_{01}^{'''} & 
\binom{9}{4}x_{01}^{(5)}
\\
0 & 
0 & 
0 & 
\binom{5}{5}x_{02} &
0 & 
\binom{7}{5}x_{01}^{''} & 
\binom{9}{5}x_{01}^{(4)}
\\
0 & 
0 & 
0 & 
0 &
0 & 
\binom{7}{6}x_{01}^{'} & 
\binom{9}{6}x_{01}^{'''}
\end{vmatrix},
\endaligned
\]
that are constructed, generally, in the following way. 
By convention, for $k, n \in \mathbb{Z}$:
\[
{\textstyle{\binom{n}{k}}}
=
0
\ \ \ \ \ \ \ \ \ \ 
\text{whenever}\ \
k\leqslant -1\,\,
\text{or}\,\,
n+1\leqslant k.
\]

First, choose the size of the square matrix, {\em e.g}
$7 \times 7$ above.
Each column has the following form:
\[
\begin{pmatrix}
x_{0i}^{(l)} \\
\binom{l}{1} x_{0i}^{(l-1)} \\
\binom{l}{2} x_{0i}^{(l-2)} \\
\vdots \\
x_{0i} \\
0 \\
\vdots
\end{pmatrix}
\eqno
{\scriptstyle{(N\,\geqslant\,i\,\geqslant\,1)}},
\] 
with as many lines as the chosen size.
With descending $i = N, N-1, \dots, 1$
(from left to right),
the matrix incorporates $N$
blocks of such columns, 
with the block number $i$ 
involving only the variables
$x_{0i}, x_{0i}', x_{0i}'', \dots$.
Two further conditions must be satisfied: 
{\small (C1)} In each block number $i$, 
the derivative orders $x_{0i}^{(l_1)}$, $x_{0i}^{(l_2)}$, 
$x_{0i}^{(l_3)}$, 
$\dots$ appearing on the
top (on the first line),
must be stricly increasing $l_1 < l_2 < l_3 < \cdots$
when going from left to right;
{\small (C2)} 
The diagonal terms of the obtained matrix must all be nonzero.

Of course, such a matrix is entirely characterized by the
derivatives $x_{0i}^{(l)}$
appearing on its first line. We denote by
$\mathcal{B}^{N+}$ the set of such possible first lines, and for each
such line $\underline{\underline{\alpha}}$ in $\mathcal{B}^{N+}$,
we denote
by $\Delta_0(\underline{\underline{\alpha}})$ the associated matrix.

We can now express an alternative, interesting basis for E,
which will be of crucial use.

\begin{Theorem}
The family:
\[
\Big\{
\det\,
\Delta_0(\underline{\underline{\alpha}})
\Big\}
_{\underline{\underline{\alpha}}  \in \mathcal{B}^{N+}}
\]
is a basis of $E = \mathbb{C} [x_{0i}, x_{0i}', x_{0i}'', \dots]$.
\end{Theorem}

For such a determinant, if $M(\underline{\underline{\alpha}})$ is the
highest derivative that appears in it, 
then we will show that in any other
affine chart $U_i = \{X_i \neq 0\}$,
the pole with the highest order has order
$M(\underline{\underline{\alpha}})+1$. 
A simple combinatoric argument, developed in Lemma \ref{cardinal_famille_libre}, shows that
the number of such $\det\,
\Delta_0
(\underline{\underline{\alpha}})$
with 
$M(\underline{\underline{\alpha}}) \leqslant d-1$
equals $(N+1)^d$.

\begin{Theorem}
Let $d$ be some fixed integer. 
Then, for every $k \geqslant d-1$ and every
$\underline{\underline{\alpha}} \in \mathcal{B}^{N+}$ 
with $M(\underline{\underline{\alpha}}) \leqslant d-1$, one has:
\[
\det\,
\Delta_0(\underline{\underline{\alpha}})
\in
H^0\big(
\mathbb{P}^N,\,
E^{\textbf{GG}}_{k, \infty}(d)
\big).
\]
In addition, these $\det\, \Delta_0
(\underline{\underline{\alpha}})$ are linearly independent,
hence:
\[
\dim\,
H^0
\big(
\mathbb{P}^N,\,
E^{\textbf{GG}}_{k, \infty}(d)
\big)
\geqslant
(N+1)^d.
\]
\end{Theorem}

We have thus at least found a free family of
$(N+1)^d$ global polynomial sections.

The second part consists in showing that there are no more polynomial
sections. Every determinant of our basis can be written in the
affine chart
$U_1 = \{X_1 \neq 0\}$ 
as a sum of rational terms that have the form:
\[
\frac{M}{x_{11}^s},
\]
where $M$ is a monomial in the other variables $x_{12}, \dots, x_{1N},
x_{11}', \dots$, and $s \geqslant 0$. Based on the following ordering on the variables:
\[
\aligned
&x_{1N}
\prec
x_{1N}'
\prec
x_{1N}''
\prec
\cdots \\
\cdots
\prec
&\,x_{1N-1}
\prec
x_{1N-1}'
\prec
\cdots \\
&\quad \quad \vdots \\
\cdots
\prec
&\,x_{11}
\prec
x_{11}'
\prec
\cdots,
\endaligned
\] 
let us define an ordering on
those rational terms.
To compare two rational terms $\frac{M}{x_{11}^s}$ and
$\frac{M'}{x_{11}^{s'}}$, we start by comparing $s$ and $s'$. If $s >
s'$, then $\frac{M}{x_{11}^s} \prec \frac{M'}{x_{11}^{s'}}$. If
$s=s'$, then we compare the biggest variables of $M'$ and $M$. The
term with the biggest variable is the biggest term. If the largest
variables coincide, we compare the second largest, and so on.

We establish our theorem thanks to the following
key observation: different
determinants $\det\,
\Delta_0(\underline{\underline{\alpha}})$ 
have different smallest terms in the chart $U_1$. 
Arguments and proofs 
appear in Section~{\ref{cas_unidimensionel}} 
in the simpler case of dimension $N = 1$ 
and in Section~{\ref{cas_general}} in any dimension $N \geqslant 1$.
There is a little complication when searching for the
smallest term after a change of chart in dimension $N \geqslant 2$.

The free family of global sections we have found is thus a basis.

\begin{Theorem}
For every $d \geqslant 1$ and
every $k \geqslant d-1$, the family:
\[
\Big\{
\det\,
\Delta_0
(\underline{\underline{\alpha}})
\colon\,\,
M(\underline{\underline{\alpha}}) \leqslant d-1 
\Big\},
\]
of cardinal $(N+1)^d$, is a basis of 
$H^0\big(\mathbb{P}^N,\, \oplus_{n=1}^{\infty} 
E_{k,n}^{\text{GG}} \otimes \mathcal{O}(d)\big)$.
\end{Theorem}

The whole scheme of our arguments,
based on manipulations of concrete explicit determinants,
is rather different from the one
developed by Etesse in~{\cite{Etesse-2023},
based on (advanced) Representation Theory.

\bigskip\noindent
\textbf{Acknowledgements.} 
The first author would like to thank
the second one for guidance 
throughout this research.  
The second author would like to mention that, 
after his many mathematical experiments on Maple in small dimensions
shared for inspiration, 
it was the first author who was able to devise key arguments 
and clever proofs valid in general dimensions.

Both authors would also like to thank Paris-Saclay University,
and the Master 2 `{\sl Arithmétique, Algèbre, et Géométrie}' as well.

%%%%%%%%%%%%%%%%%%%%%%%%%%%%%%%%%%%%%%%%%%%%%%%%%%%%%%%%%%%%%%%%%%%%%%
\section{The projective space}\label{the-projective-space}
%%%%%%%%%%%%%%%%%%%%%%%%%%%%%%%%%%%%%%%%%%%%%%%%%%%%%%%%%%%%%%%%%%%%%%

In this section $N$ will denote the dimension of the projective space
that we will construct in the following way: on
$(\mathbb{C}^{N+1})^{*}$, we consider the equivalence relation $\sim$
such that:
\[
(X_0,X_1,\dots ,X_N) 
\sim 
(\lambda X_0,\lambda X_1,\dots , \lambda X_N) 
\]
for all $(X_0,X_1,\dots ,X_N) \in (\mathbb{C}^{N+1})^{*}$ and $\lambda
\in \mathbb{C}^*$.

The quotient space $\mathbb{P}^N(\mathbb{C}) = (\mathbb{C}^{N+1})^{*}
/ \sim$ is called the projective space of dimension $N$. It is a
topological space endowed with the quotient topology.

For a point $(X_0,X_1,\dots ,X_N)$ in $\mathbb{C}^{N+1}$, we will
denote its equivalence class in the projective space by
$[X_0\colon X_1 \colon \dots \colon X_N]$. 
Every point $x$ in $\mathbb{P}^N(\mathbb{C})$
can be represented in this way: $x = 
[X_0\colon X_1 \colon \dots \colon X_N]$. The
elements in the $(N+1)$-uple $(X_0,X_1,\dots ,X_N)$ are called the
{\sl homogeneous coordinates} of the point $x$.  For every $i =
0,1,\dots ,N$, the open subsets:
\[
U_i
= 
\Big\{ [X_0\colon X_1\colon \dots \colon
X_N] \in \mathbb{P}^N(\mathbb{C}) 
\colon\,\, 
X_i \neq 0 \Big\}.
\]
along with the maps:
\begin{center}
$\begin{array}{ccccl}
\phi_i & : & U_i & 
\to & \mathbb{C}^N 
\\
& & [X_0 \colon X_1 \colon \dots \colon X_N] & 
\mapsto & (\frac{X_0}{X_i},\dots ,\frac{X_N}{X_i}) 
\\
\end{array}$,
\end{center}
endow $\mathbb{P}^N(\mathbb{C})$ with a structure of an
$N$-dimensional manifold.

\begin{Example}
When $N=1$, a point $p$ with homogeneous coordinates 
$[X_0 \colon X_1]$,
where $X_0 \neq 0$ and $X_1 \neq 0$, can be read in the chart
$(U_0,\phi_0)$ and in the chart $(U_1,\phi_1)$.  We denote by $x$ the
affine coordinate of $p$ read in the chart $(U_0,\phi_0)$, i.e $x =
\phi_0(p)$. Then one easily sees that $\phi_1(p) = \frac{1}{x}$, so
that the change of charts is given by:
\begin{center}
$\begin{array}{ccccl}
\phi_{1,0} & : & \mathbb{C} & \to & \mathbb{C}
\\
& & x & \mapsto & 1/x 
\\
\end{array}.$
\end{center}
There are in $\mathbb{P}^1(\mathbb{C})$ two distinguished points,
namely $\mathbf{0} = [1:0]$ and $\mathbf{\infty} = [0:1]$. The first
one has $0$ as affine coordinate in the chart $(U_0,\phi_0)$, and
corresponds to the point at infinity in the chart$(U_1,\phi_1)$.
\end{Example}

In the general case, a point $p$ in 
$\mathbb{P}^N(\mathbb{C})$ has $N
+ 1$ homogeneous coordinates $[X_0
\colon X_1\colon \dots \colon X_N]$,
and on each open set $U_i = \{ X_i \neq 0\}$
for $i = 0, 1, \dots, N$, 
the point $p$ has an $N$-tuple of affine coordinates.
For $i \in \llbracket 0, N
\rrbracket$, the affine coordinates of the point $p$ in the chart
$U_i$ will be denoted $(x_{i1}, x_{i2}, \dots, x_{iN})$, where:
\[
\aligned
x_{ii} &= \frac{X_0}{X_i}, \\
x_{ij} &= \frac{X_j}{X_i}
\ \ \text{if} \ j \neq i.
\endaligned
\]

We can now, in terms of the affine coordinates, explicitly describe
the changes of charts. If $i \neq j$, then:
\[
\aligned
x_{ii}
&=
\frac{X_0}{X_i}
=
\frac{x_{jj}}{x_{ji}}, \\
x_{ij}
&=
\frac{X_j}{X_i}
=
\frac{1}{x_{ji}},
\endaligned
\]
and, for all $t \not\in \{i, j\}$,
\[
x_{it}
=
\frac{X_t}{X_i}
=
\frac{x_{jt}}{x_{ji}}.
\]

%%%%%%%%%%%%%%%%%%%%%%%%%%%%%%%%%%%%%%%%%%%%%%%%%%%%%%%%%%%%%%%%%%%%%%
\subsection{The line bundle $\mathcal{O}(d)$}
%%%%%%%%%%%%%%%%%%%%%%%%%%%%%%%%%%%%%%%%%%%%%%%%%%%%%%%%%%%%%%%%%%%%%%

Over each open set $U_i$ of the projective space, 
consider the trivial line bundle:
\begin{center}
\begin{tikzcd}
U_i \times \mathbb{C} \arrow[d, "\pi"] 
\\
U_i
\end{tikzcd}.
\end{center}
Now fix an integer $d \in \mathbb{N}$. We glue these trivial line
bundles together with the transition maps $G_{ji}
([X_0 \colon X_1 \colon \dots \colon X_N])
= \big( \frac{X_i}{X_j} \big)^d$, resulting in the following
commutative diagram:
\begin{center}
$\begin{tikzcd}[column sep=small]
(p,v) \in U_i \times\mathbb{C} \arrow{rr} & & \Big( 
p,\Big(\frac{X_i}{X_j}\Big)^d v\Big) \in  U_j \times\mathbb{C} \\
& p \in U_i \arrow[ul] \arrow[ur] &
\end{tikzcd}$
\end{center}
One can easily verify that the cocycle condition:
\begin{center}
$G_{ik}G_{kj}G_{ji} = 1$
\end{center}
is satisfied, so the transition maps correctly define a line bundle.

%%%%%%%%%%%%%%%%%%%%%%%%%%%%%%%%%%%%%%%%%%%%%%%%%%%%%%%%%%%%%%%%%%%%%%
\subsection{The bundle of k-jets}
%%%%%%%%%%%%%%%%%%%%%%%%%%%%%%%%%%%%%%%%%%%%%%%%%%%%%%%%%%%%%%%%%%%%%%

Let $k$ be an integer. Over each open set $U_i$ we consider
the trivial vector bundle of dimension $N \times k$:
\begin{center}
\begin{tikzcd}
U_i \times \mathbb{C}^{N \times k} \arrow[d, "\pi"] \\
U_i
\end{tikzcd}.
\end{center}
We will glue these trivializations together, but we will first need to
discuss how to interpret the elements in the fibers as $k$-jets. Let
$p$ be an abstract point on the projective space $\mathbb{P}^N$. This
point $p$ can be read in different charts. Suppose that $p \in U_i
\cap U_j$, and write its affine coordinates in the two local charts:
\[
\phi_i(p) 
= 
(x_1,\dots,x_N) \in \mathbb{C}^N, \quad
\phi_j(p) 
= 
(y_1,\dots,y_N) \in \mathbb{C}^N.
\]
Now, consider some holomorphic curve passing through $p$, i.e. a map
$\gamma$:
\[
\gamma : D(0,R) 
\longrightarrow 
\mathbb{P}^N,
\]
with $\gamma (0) = p$, and such that:
\[
\phi_i \circ \gamma: \mathbb{C} \longrightarrow \mathbb{C}^N
\]
is holomorphic.  We may consider the derivatives of $\phi_i \circ
\gamma$ and denote them adequately:
\begin{align*}
\Big( \phi_i \circ \gamma \Big)'
&= 
(x_1',x_2', \ldots, x_N'), \\
\Big( \phi_i \circ \gamma \Big)''
&= 
(x_1'',x_2'', \ldots, x_N''), \\
&\,\,\:
\vdots \\
\Big( \phi_i \circ \gamma \Big)^{(k)}
&= 
(x_1^{(k)},x_2^{(k)}, \ldots, x_N^{(k)}).
\end{align*}

The vector $(x_1', x_2', \dots, x_N^{(k)})$ is called a $k$-jet and is
in the fiber $\mathbb{C}^{N \times k}$.  We use the letter $x$ because
we chose to `read' the curve in the local chart $U_i$. Of course, we
could have done the same thing with the derivatives of $\phi_j \circ
\gamma$, which would have given us a different vector.
\begin{align*}
\Big( \phi_j \circ \gamma \Big)'
&= 
(y_1',y_2', \ldots, y_N'), \\
\Big( \phi_j \circ \gamma \Big)''
&= 
(y_1'',y_2'', \ldots, y_N''), \\
&\,\,\:\vdots \\
\Big( \phi_j \circ \gamma \Big)^{(k)}
&= 
(y_1^{(k)},y_2^{(k)}, \ldots, y_N^{(k)}).
\end{align*}

If $\Psi_t$ is the change-of-chart map:
\[
y_t = \Psi_t(x_1, \dots, x_N),
\]
then one can recover the $l$-jet $y^{(l)}_t$ by differentiating the
previous equation:
\[
y^{(l)}_t
=
\Psi_t \Big(x_1, \dots, x_N\Big)^{(l)}.
\]
Here, the derivatives are computed by considering $x_1, \dots, x_N$ as
holomorphic functions.

This construction can be summarized in the following diagram:
\[
\begin{tikzcd}
& (x_1', \ldots, x_N^{(k)}) \in \pi^{-1}(p) \ni (y_1', \ldots, y_N^{(k)})
\arrow[ddl]
\arrow[ddr]
\arrow[dd, "\pi"] & \\ \\
(x_1, \ldots, x_N) \in \mathbb{C}^N
& p \in \mathbb{P}^N \arrow[r, "\phi_i"] \arrow[l, "\phi_j"]
& \mathbb{C}^N \ni (y_1, \ldots, y_N) \\
& \mathbb{C} \arrow[u, "\gamma"]
\end{tikzcd} 
\]

\begin{Example}
Recall that in dimension $N=1$,
the change of charts is given by
\[
\phi_1 \circ \phi_0^{-1}: t \longmapsto \frac{1}{t}.
\]
Let $p$ be an abstract point in $\mathbb{P}^1$, and denote by $x$ its
affine coordinate in $(U_0, \phi_0)$ and by $y$ its affine coordinate
in $(U_1, \phi_1)$. Thus $y=1/x$.  Let $\gamma$ be a holomorphic
curve passing through $x$ and denote by $x', x'', \ldots, x^{(k)} \in
\mathbb{C}$ the derivatives of $\phi_0 \circ \gamma$:
\[
x' 
= 
\Big( \phi_0 \circ \gamma \Big)', \quad x'' 
= 
\Big( \phi_0 \circ \gamma \Big)'', \ \ldots, \ x^{(k)} 
= 
\Big( \phi_0 \circ \gamma \Big)^{(k)}.
\]
Similarly, we shall denote the derivatives of $\phi_1 \circ \gamma$ by $y', y'', \ldots, y^{(k)}$:
\[
y' 
= 
\Big( \phi_1 \circ \gamma \Big)', \quad y'' 
= 
\Big( \phi_1 \circ \gamma \Big)'', \ \ldots, \ y^{(k)} 
= 
\Big( \phi_1 \circ \gamma \Big)^{(k)}.
\]
Now, using the classic formula $(f \circ g)' = f'(g)g'$ with:
\begin{align*}
g &= x, \\
\Psi &= \phi_1 \circ \phi_0^{-1},
\end{align*}
one computes:
\begin{align*}
y'
&=
(\Psi \circ g)' \\
&=
\frac{1}{x^2}\Big(-x'\Big). \\ \\
y''
&=
(\Psi \circ g)'' \\
&=
\frac{2}{x^3}x'x' - \frac{1}{x^2}x''.
\end{align*}
\end{Example}

\begin{Example}
In the case $k=1$, 
the $k$-jets (which are now $1$-jets) are given by a simple
derivative. The transition map from $U_i$ to $U_j$
is:
\[
G_{ji}(p): 
\Big( \phi_i \circ \gamma \Big)'(0) 
\longmapsto 
\Big( \phi_j \circ \gamma \Big)'(0).
\]
The formula for the differential of a composition gives us:
\begin{center}
$\Big( \phi_j \circ \gamma \Big)'(0) = D_0(\phi_i \circ \phi_j^{-1})\Big(\phi_i \circ \gamma \Big)'(0)$.
\end{center}
In particular, we see that the transition maps are linear, hence the
bundle of $1$-jets is a vector bundle.
The vector bundle of $1$-jets is actually the tangent bundle.
\end{Example}

%%%%%%%%%%%%%%%%%%%%%%%%%%%%%%%%%%%%%%%%%%%%%%%%%%%%%%%%%%%%%%%%%%%%%%
\subsection{The Green-Griffiths bundle}
%%%%%%%%%%%%%%%%%%%%%%%%%%%%%%%%%%%%%%%%%%%%%%%%%%%%%%%%%%%%%%%%%%%%%%

In this section we will construct the \textit{Green-Griffiths} bundle,
which is, contrary to the bundle of $k$-jets, a vector bundle.

Let $p \in U_i \cap U_j$ a point in $\mathbb{P}^N$, that we might
express in its affine coordinates in the two charts:
\[
\phi_i(p) = (x_1, x_2, \dots, x_N) 
\ \ \ \text{and} \ \ \
\phi_j(p) = (y_1, y_2, \dots, y_N).
\]
For all $t \in \llbracket 1, N \rrbracket$, we have a change of charts
map:
\[
y_t 
=
\Psi_t(x_1, x_2, \dots, x_N),
\]
which one can differentiate 
to get the change of charts of the bundle of $k$-jets:
\[
y^{(l)}_t 
=
(\Psi_t(x_1, x_2, \dots, x_N))^{(l)}.
\]
We saw that for $l \geqslant 2$, this map fails to be linear.

In this section (and only in this section), we define the weight of a
monomial $x^{(a_1)}x^{(a_2)}\cdots x^{(a_n)}$ with $a_1, a_2, \dots,
a_n \geqslant 1$ to be $a_1 + a_2 + \cdots + a_n$.

For every $k \geqslant 1$ and $n \geqslant 1$, we define the fiber
over the trivializing set $U_i$ of the {\sl Green-Griffiths bundle of
order $k$ and degree $n$} of the projective space
$\mathbb{P}^N(\mathbb{C})$ to be the polynomial space generated by the
monomials of weight $n$ in $\mathbb{C}[x_1', \dots, x_N', x_1'',
\dots, x_N^{(k)}]$. An element in the fiber can thus be
written as:
\[
P
= 
\sum_{a_{11}+\cdots+a_{1N}+2a_{21}+\cdots+ka_{kN} = n}
c_{a_{11},\dots,a_{1N},a_{21},\ldots,a_{kN}}
(x_1')^{a_{11}}
\ldots
(x_N')^{a_{1N}}
(x_1'')^{a_{21}}
\ldots
(x_N^{(k)})^{a_{kN}}.
\]

\begin{Example}
We give the monomials of weight $n = 2, 3, 4$ 
for the one-dimensional projective space $\mathbb{P}^1(\mathbb{C})$:
\begin{itemize}

\item 
If $n = 2$, there are two such monomials, namely:
\[
x'x', \ \ \ \ \ \ x''.
\]
\item 
If $n = 3$, there are three such monomials, which are:
\[
x'x'x', \ \ \ \ \ \ x'x'', \ \ \ \ \ \ x'''.
\]
\item 
If $n = 4$, there are five such monomials:
\[
x'x'x'x', \ \ \ \ \ \ x'x'x'', \ \ \ \ \ \ 
x''x'', \ \ \ \ \ \ x'x''', \ \ \ \ \ \ x''''.
\]
\end{itemize}

\end{Example}

We shall now describe the changes of charts.  Let $U_j$ be some other
trivializing open set, in which the affine coordinates are denoted
$(y_1, \dots, y_N)$.  Let $i_1, \dots, i_s$ be some indices
in $\llbracket 1, N \rrbracket$, and $m_1, \dots, m_s$ some
integers with $m_1 + \cdots + m_s = n$. The monomial
$y_{i_1}^{(m_1)}\cdots y_{i_s}^{(m_s)}$ is an element
in the fiber over the set $U_j$ and the change of charts of the
Green-Griffiths bundle is given by:
\[
y_{i_1}^{(m_1)}\cdots y_{i_s}^{(m_s)}
=
(\Psi_{i_1}(x_1, \dots, x_N))^{(m_1)}
\dots
(\Psi_{i_s}(x_1, \dots, x_N))^{(m_s)}.
\]

The Green-Griffiths bundle of order $k$ and degree $n$ is denoted
$E^{\text{GG}}_{k, n}$.

\begin{Theorem}
{\rm \cite{Demailly-2015, Merker-2015}}
The Green-Griffiths bundle is a vector bundle.
\end{Theorem}

The proof of the linearity of the changes of charts can be found 
in~\cite{Demailly-2015, Merker-2015}.

For every $k \geqslant 1$, we define the external direct sum:
\[
E^{\text{GG}}_{k, \infty} 
= 
\bigoplus_{n = 0}^{\infty} E^{\text{GG}}_{k,n}.
\]

A local holomorphic section, over some trivializing open $U_i$, can be
written as a convergent series:
\[
\sum 
c_{\centersmallbullet}\, 
x_1^{a_{01}}
\cdots 
x_N^{a_{0N}} (x_1')^{a_{11}} 
\cdots 
(x_N^{(k)})^{a_{kN}}.
\]
Every non-constant term of this series will yield, after a change of
charts, a pole in some variable $x_j$. To compensate these poles, so
that the local sections can extend holomorphically 
to global sections, we
consider the {\sl twisted Green-Griffiths bundles}:
\[
E^{\text{GG}}_{k, \infty}(d)
=
E^{\text{GG}}_{k, \infty} \otimes \mathcal{O}(d)
\ \ \ \forall d \in \mathbb{N}.
\]

If $(x_1, \dots, x_N)$ and $(y_1, \dots, y_N)$ are the
affine coordinates of some points in two distinct trivializing open
sets $U_i$ and $U_j$, and if $y^{(m_1)}_{t_1}\cdots
y^{(m_l)}_{t_l}$ is some element in the bundle $E^{\text{GG}}_{k,
\infty}(d)$, then the changes of charts can be written as:
\[
y^{(m_1)}_{t_1}\cdots y^{(m_l)}_{t_l}
=
x^d_j(\Psi_{t_1}(x_1, x_2, \dots, x_N))^{(m_1)}
\cdots
(\Psi_{t_l}(x_1, x_2, \dots, x_N))^{(m_l)}.
\]

%%%%%%%%%%%%%%%%%%%%%%%%%%%%%%%%%%%%%%%%%%%%%%%%%%%%%%%%%%%%%%%%%%%%%%
\subsection{Global sections of the twisted Green-Griffiths bundle}
%%%%%%%%%%%%%%%%%%%%%%%%%%%%%%%%%%%%%%%%%%%%%%%%%%%%%%%%%%%%%%%%%%%%%%

A holomorphic section of the twisted Green-Griffiths bundle $E_{k,
\infty}^{\text{GG}}(d)$ of degree $d$ can be written as a family of
local sections $(s_0, s_1, \dots, s_N)$ over the trivializing open
sets $U_0, U_1, \dots, U_N$.

For every $i = 0, 1, \dots, N$, the local section $s_i$ is in the
affine variables $x_{i1}, x_{i2}, \dots, x_{iN}$, and the local
sections satisfy the changes of charts conditions:
\[
s_i(x_{i1}, x_{i2}, \dots, x_{iN})
=
x_{ij}^d s_j(x_{j1}, x_{j2}, \dots, x_{jN})
\ \ \ \forall i, j.
\]

Furthermore, a local holomorphic section $s_i$ can be written as a
convergent series:
\[
s_i
=
\sum_{k_1 \leqslant k_2 \leqslant \cdots \leqslant k_l}
\sum_{a_1, a_2, \cdots, a_l}
H^{k_1, k_2, \dots, k_l}_{a_1, a_2, \dots, a_l} \
x_{ia_1}^{(k_1)}
x_{ia_2}^{(k_2)}
\cdots
x_{ia_l}^{(k_l)}.
\]

We will say that a global section $s$ is polynomial if the local
sections $s_0, s_1, \dots, s_N$ are all polynomials. The space of
polynomial sections will be denoted
$H^0_{\text{pol}}(\mathbb{P}^N(\mathbb{C}), E_{k,
\infty}^{\text{GG}}(d))$.

\begin{Remark}
It can be shown, although it will not be necessary here, that every
holomorphic section is in fact polynomial:
\[
H^0_{\text{pol}}\big(\mathbb{P}^N(\mathbb{C}), 
E_{k, \infty}^{\text{GG}}(d)\big)
=
H^0(\mathbb{P}^N\big(\mathbb{C}), E_{k, \infty}^{\text{GG}}(d)\big).
\]
\end{Remark}

%%%%%%%%%%%%%%%%%%%%%%%%%%%%%%%%%%%%%%%%%%%%%%%%%%%%%%%%%%%%%%%%%%%%%%
\section{The one-dimensional case}\label{cas_unidimensionel}
%%%%%%%%%%%%%%%%%%%%%%%%%%%%%%%%%%%%%%%%%%%%%%%%%%%%%%%%%%%%%%%%%%%%%%

In this section we will prove that the dimension of the space
$H^0\big( \mathbb{P}^1(\mathbb{C}), E^{\text{GG}}_{k,
\infty}(d) \big)$ is $2^d$ when $k \geqslant d-1$. We will work in
affine coordinates to describe local sections on the open $U_0$. The
natural basis for the local sections is the family of the monomials in
the variables $x, x', x'', \dots$. We will introduce another basis,
given by determinants, that will be more useful to identify the local
sections that extend holomorphically to global sections.

%%%%%%%%%%%%%%%%%%%%%%%%%%%%%%%%%%%%%%%%%%%%%%%%%%%%%%%%%%%%%%%%%%%%%%
\subsection{The polynomial space of local sections}
\label{subsection-polynomial-space}
%%%%%%%%%%%%%%%%%%%%%%%%%%%%%%%%%%%%%%%%%%%%%%%%%%%%%%%%%%%%%%%%%%%%%%

Let $\mathcal{A}$ be the set of sequences 
with a finite number of terms $(a_i)_{i \in \llbracket1, n\rrbracket}$
which are in $\mathbb{N}$ and weakly increasing: 
\[
0 \leqslant a_1 \leqslant a_2 \leqslant 
\dots 
\leqslant a_n.
\]
We will often denote the sequence $(a_i)_{i \in \llbracket1, n\rrbracket}$ by $\underline{a}$. 

Furthermore, let $\mathcal{B}$ be the set of sequences with a finite number of terms $(\alpha_i)_{i \in \llbracket1, n\rrbracket}$ which are in $\mathbb{N}$ and strictly increasing:
\[
0 \leqslant \alpha_1 < \alpha_2 <
\dots
< \alpha_n.
\]
Such a sequence $(\alpha_i)_{i \in \llbracket 1, n \rrbracket}$ 
shall be denoted $\underline{\alpha}$.
The sets $\mathcal{A}$ and $\mathcal{B}$ 
both contain the particular sequence $\emptyset$ having zero terms.

There is a natural correspondence between $\mathcal{A}$ 
and $\mathcal{B}$ given by the map:
\[
\begin{array}{ccccl}
\tau & : & \mathcal{B} & \to & 
\quad \quad \quad \quad \quad \quad 
\mathcal{A} 
\\
& & (\alpha_i)_{i \in \llbracket 1, n \rrbracket} & \mapsto & (\alpha_i - i + 1)_{i \in \llbracket 1, n \rrbracket} 
=:
(a_i)_{i \in \llbracket 1, n \rrbracket}
\\
\end{array}.
\]
\begin{Lemma}
The map $\tau$ is a bijection.
\end{Lemma}
\begin{proof}
We can explicitly give its inverse:
\[
\begin{array}{ccccl}
\tau^{-1} & : & \mathcal{A} & \to & 
\quad \quad \quad \quad \quad \quad 
\mathcal{B} 
\\
& & (a_i)_{i \in \llbracket 1, n \rrbracket} & \mapsto & (a_i + i - 1)_{i \in \llbracket 1, n \rrbracket} 
=:
(\alpha_i)_{i \in \llbracket 1, n \rrbracket}
\\
\end{array}.
\qedhere
\]
\end{proof}
Consider the space of polynomials:
\[
E 
= 
\mathbb{C}[x,x',x'',\dots].
\]
A polynomial in $E$ can be written as a linear combination of
monomials, and these monomials can themselves be written as:
\[
x^{(a_1)}x^{(a_2)}\cdots x^{(a_n)},
\]
with $a_1, \dots , a_n \in \mathbb{N}$ and, since the order of the
factors does not matter, we may as well suppose that:
\[
0 \leqslant a_1 \leqslant a_2 \leqslant \dots \leqslant a_n,
\]
so that we have a sequence in $\mathcal{A}$.
We will thus, for all sequences $\underline{a} \in \mathcal{A}$, 
use the following notation:
\[
x^{(\underline{a})} = x^{(a_1)}x^{(a_2)}\cdots x^{(a_n)},
\]
with the convention:
\[
x^{(\emptyset)}
=
1.
\]

\begin{Fact}
The family
\[
\{ x^{(\underline{a})}\}
_
{\underline{a} \in \mathcal{A}}
\]
is the canonical basis of $E$.
\end{Fact}

%%%%%%%%%%%%%%%%%%%%%%%%%%%%%%%%%%%%%%%%%%%%%%%%%%%%%%%%%%%%%%%%%%%%%%
\subsection{Weighting the monomials} \label{poids_monomes}
%%%%%%%%%%%%%%%%%%%%%%%%%%%%%%%%%%%%%%%%%%%%%%%%%%%%%%%%%%%%%%%%%%%%%%

For some integer $a \in \mathbb{N}$, we say that the variable
$x^{(a)}$ has \emph{weight} $a+1$, because if $x$ is a function in
some variable $t$, then by using induction we may show that:
\begin{align*}
\Big(\frac{1}{x}\Big)'
&=
-\frac{x'}{x^2}
\\
\Big(\frac{1}{x}\Big)''
&=
\frac{2x'x' - xx''}{x^3}
\\
&\vdots \\
\Big(\frac{1}{x}\Big)^{(a)}
&=
\frac{\centersmallbullet}{x^{a+1}},
\end{align*}
where $\centersmallbullet$ is some polynomial in $E$.  It is the weight we have
to compensate when changing charts in the Green-Griffiths
bundle. Naturally, the weight of a product will be the sum of the
weights. We shall formalize this by defining the weight function on
the set $\mathcal{A}$:
\[
\begin{array}{ccccl}
w_{\mathcal{A}} & : & \mathcal{A} & \to & 
\quad \quad \quad \quad \quad 
\mathbb{N} 
\\
& & (a_i)_{i \in \llbracket 1, n \rrbracket} & 
\mapsto & (a_1 + 1) + (a_2 + 1) + \dots + (a_n + 1).
\end{array}
\]

We can also define a weight function on the set $\mathcal{B}$ by
saying that the weight of some element $\underline{\alpha}$ in
$\mathcal{B}$ is the weight of its alter ego
$\tau(\underline{\alpha})$ in $\mathcal{A}$:
\[
\begin{array}{cccccl}
w_{\mathcal{B}} & : & \mathcal{B} & \to & &
\quad \quad \quad 
\mathbb{N} 
\\
& & (\alpha_i)_{i \in \llbracket 1, n \rrbracket} & \mapsto &
w_{\mathcal{B}}(\underline{\alpha})
&:=
w_{\mathcal{A}}(\tau(\underline{\alpha})) \\
& & & & &= (\alpha_1 + 1) + \alpha_2 + \dots + (\alpha_n + 2 - n).
\end{array}
\]
For every integer $p \in \mathbb{N}$, 
we can now define the finite sets:
\[
\mathcal{A}_p = \Big\{ \underline{a} \in \mathcal{A} 
\colon\,\,
w_{\mathcal{A}}(\underline{a}) = p \Big\},
\]
and:
\[
\mathcal{B}_p 
= 
\Big\{ \underline{\alpha} 
\in 
\mathcal{B} 
\colon\,\, w_{\mathcal{B}}(\underline{\alpha}) 
= 
p 
\Big\}.
\]
These two sets have the same cardinality, for $\tau$ induces a
bijection between them.

Finally, we define the subspace of $E$ of polynomials of \emph{weight}
$p \in \mathbb{N}$:
\[
E_p = \text{Vect}\Big\{x^{(\underline{a})}\Big\}_{\underline{a} \in \mathcal{A}_p},
\]
which has a natural basis given by the monomials:
\[
\Big\{
x^{(\underline{a})}
\Big\}
_
{\underline{a} \in \mathcal{A}_p},
\]
so that one has the triple equality:
\begin{equation}
\label{dim-Ep-Card-Ap-Card-Bp}
\dim E_p
=
\text{Card}(\mathcal{A}_p)
=
\text{Card}(\mathcal{B}_p).
\end{equation}

%%%%%%%%%%%%%%%%%%%%%%%%%%%%%%%%%%%%%%%%%%%%%%%%%%%%%%%%%%%%%%%%%%%%%%
\subsection{Ordering the monomials} \label{ordre_monomes}
%%%%%%%%%%%%%%%%%%%%%%%%%%%%%%%%%%%%%%%%%%%%%%%%%%%%%%%%%%%%%%%%%%%%%%

In this section we will define an order on the monomials. We start by
defining the following order on the variables:
\[
1 \prec x \prec x' \prec x'' \prec \dots
\] 
which we extend to the monomials in the following way: let
$\underline{a}$ and $\underline{b}$ be two sequences in $\mathcal{A}$,
and consider the monomials:
\begin{align*}
x^{(\underline{a})} 
&=
x^{(a_1)}x^{(a_2)}x^{(a_3)}\dots x^{(a_n)} \\
x^{(\underline{b})} 
&=
x^{(b_1)}x^{(b_2)}x^{(b_3)}\dots x^{(b_m)},
\end{align*}
with $a_1 \leqslant a_2 \leqslant \dots \leqslant a_n$ and $b_1
\leqslant b_2 \leqslant \dots \leqslant b_m$. 

To compare them, we first look at their biggest derivative.

\begin{itemize}

\item If $a_n < b_m$, then
\[
x^{(\underline{a})} \prec x^{(\underline{b})}.
\]

\item If $a_n = b_m$, then we compare
\[
x^{(a_1)}x^{(a_2)}\dots x^{(a_{n-1})} 
\quad 
\text{and} \quad x^{(b_1)}x^{(b_2)}\dots x^{(b_{m-1})}.
\]
\end{itemize}

By induction, this defines a total ordering on the monomials, 
that can be represented by:
\begin{align*}
1 \prec x& \prec xx \prec xxx \prec \dots \\
\dots \prec x'& \prec x'x \prec x'xx \prec \dots \\
\dots \prec x'x'& \prec x'x'x \prec x'x'xx \prec \dots \\
 & \dots  \\
 & \dots  \\
\dots \prec x''& \prec x''x \prec x''xx \prec \dots \\
\dots \prec x''x'& \prec x''x'x \prec x''x'xx \prec \dots.
\end{align*}
Each polynomial in $E = \mathbb{C}[x,x',x'', \dots]$ has a minimal
monomial. In the next section, we will define polynomials in $E$ as
determinants of some matrices, and finding the minimal monomial can be
a difficult task.

However, if one condition on the matrix is fulfilled, the minimal
monomial of its determinant will be handed on a silver platter, which
is the diagonal.

\begin{Theorem} \label{minth}
Let $M$ be a matrix of size $k \times k$ such that, for all $i \in
\llbracket 1,k \rrbracket$ and $j \in \llbracket 1,k \rrbracket$, the
entry $m_{ij}$ on the $i^{\text{th}}$ line and $j^{\text{th}}$ column
is a multiple of a variable in $\{ x,x',x'',\dots \}$:
\[
m_{ij}
=
c_{ij}\chi_{ij}
\]
with $c_{ij} \in \mathbb{C}$ and $\chi_{ij} \in \{x, x', x'', \dots\}$. 

Suppose also that the diagonal terms are non zero:
\[
c_{ii}
\neq
0
\quad \quad \quad \quad \quad \quad 
\text{for all } i \in \llbracket 1,k \rrbracket,
\]
and that, if for some $i,j \in \llbracket 1,k \rrbracket$ one has
$c_{ij} \neq 0$, then every entry $m_{lm}$ which is on the upper right
part of $m_{ij}$ (that means $l \leqslant i$ and $m \geqslant j$)
satisfies:
\[
c_{lm} = 0
\]
or:
\[
\chi_{ij}
< 
\chi_{lm} 
\]
Then the lowest monomial of $\det(M)$, 
when developped as a polynomial in $E$, 
is the diagonal term:
\[
m_{11}m_{22}\cdots m_{kk}.
\]
\end{Theorem}

\begin{Example}
Consider the matrix
\[
M
=
\begin{pmatrix}
x' & x'' & x''' \\
x & 2x' & 3x'' \\
0 & x & 3x'
\end{pmatrix},
\]
which satisfies the conditions of Theorem \ref{minth}. If we develop its determinant:
\[
\det M
=
6x'x'x' - 3xx'x'' - 3xx'x'' + xxx''',
\]
the smallest monomial is $6x'x'x'$, which is the diagonal term.
\end{Example}

\begin{proof}
Consider a permutation $\sigma \in \mathfrak{S}_k$ different from the
identity, and suppose that its associated monomial $M^{\sigma} = m_{1
\sigma(1)}m_{2 \sigma(2)} \cdots m_{k \sigma(k)}$ is nonzero. There
exists at least one integer $t \in \llbracket 1, k \rrbracket$ such
that $\sigma (t) \neq t$. Amongst all these integers, consider the one
such that $m_{t t}$ is the greatest.

Now, for this integer $t$, consider the rectangle submatrix of $M$
consisting of the entries that are on the top right of $m_{t t}$:
\begin{center}
$\begin{pmatrix}
m_{11} &
&
&
&
&
&
\\
&
m_{22} &
&
&
&
&
\\
&
&
\ddots &
&
&
&
\\
&
&
&
m_{t t} &
&
&
\\
&
&
&
&
\ddots &
&
\\
&
&
&
&
&
&
\end{pmatrix}$
\begin{tikzpicture}[overlay,remember picture]
\draw[color=red] (-2.5,-.3) -- (-2.5,1.5);
\draw[color=red] (-2.5,1.5) -- (-.5,1.5);
\draw[color=red] (-.5,1.5) -- (-.5,-.3);
\draw[color=red] (-.5,-.3) -- (-2.5,-.3);
\end{tikzpicture}
\end{center}

Since the part under this submatrix has more columns than rows, one
element of:
\[
m_{1 \sigma(1)}, m_{2 \sigma(2)}, \dots, m_{k, \sigma(k)}
\]
must be in the submatrix, and by the hypothesis on $t$, 
this element is
not $m_{t t}$.  This element is, by the assumptions we made on the
matrix $M$, strictly greater than $m_{t t}$. The monomial $m_{1
\sigma(1)} m_{2 \sigma(2)} \cdots m_{k \sigma(k)}$ is therefore
greater than the diagonal monomial $M^{Id}$.

We deduce that $M^{Id}$ is the smallest monomial, and since it is only
reached once for $\sigma = Id$, it does not vanish in the sum:
\[
\det(M) = \sum_{\sigma}(-1)^{\epsilon(\sigma)}M^{\sigma}.
\qedhere
\]
\end{proof}

%%%%%%%%%%%%%%%%%%%%%%%%%%%%%%%%%%%%%%%%%%%%%%%%%%%%%%%%%%%%%%%%%%%%%%
\subsection{The determinant sections}
%%%%%%%%%%%%%%%%%%%%%%%%%%%%%%%%%%%%%%%%%%%%%%%%%%%%%%%%%%%%%%%%%%%%%%

We know that, over the open set $U_0$, a holomorphic section of the
Green-Griffiths bundle can be represented as a polynomial in the space
$E = \mathbb{C}[x,x',x'',\dots]$.  The natural basis of this space is
given by the monomials:
\[
\{x^{(\underline{a})}\}_{\underline{a} \in \mathcal{A}}.
\]
In this section, we introduce a new basis whose elements are given by
determinants.

\begin{Definition}
If $\alpha$ is an integer and $\chi$ some abstract variable (which
might be, depending on the context, an affine 
coordinate variable $x$ or some
homogeneous coordinate variable $X_i$), 
we define the column of type $\alpha$ in
the variable $\chi$ to be:
\[
\mathcal{C}_{\alpha}(\chi) = \begin{pmatrix}
\chi^{(\alpha)} \\
\binom{\alpha}{1} \chi^{(\alpha - 1)} \\
\vdots \\
\binom{\alpha}{\alpha} \chi^{(0)} \\
0 \\
\vdots
\end{pmatrix},
\]
where the length of the column will depend on the context.
\end{Definition}

We shall use such columns to define square matrices as follows.

\begin{Definition} 
\label{sectionslocales}
For every sequence $\underline{\alpha} = (\alpha_1,\dots ,\alpha_n)$
in $\mathcal{B}$, define the matrix $\Delta_0(\underline{\alpha})$
to be the matrix with columns:
\[
\Delta_0(\underline{\alpha})
:=
\Big(
\mathcal{C}_{\alpha_1}(x) 
\ \Big\vert \
 \mathcal{C}_{\alpha_2}(x)
\ \Big\vert \
 \dots 
\ \Big\vert \
 \mathcal{C}_{\alpha_n}(x)
\Big).
\]

Here, the length of the columns 
is chosen so that the resulting matrix
is square, hence all columns have length $n$.
\end{Definition}

The subscript ${}_0$ 
in the notation $\Delta_0(\underline{\alpha})$ 
is there to indicate that we are describing a section 
over the open set $U_0$. 

By taking determinants of such a matrices, we get an elements in
$E$. The obtained polynomials 
have the following very important
property.

\begin{Theorem}
\label{th_base}
The family
\[
\Big\{
\det(\Delta_0(\underline{\alpha}))
\Big\}
_{\underline{\alpha}\in \mathcal{B}}
\]
is a basis of $E$.
\end{Theorem}

\begin{proof}
We will prove this theorem with the help of two 
preliminary lemmas.

\begin{Lemma}
The smallest monomial appearing in the determinant
$\det\, \Delta_0(\underline{\alpha})$ is the
diagonal one:
\[
c\,x^{(\tau(\underline{\alpha}))} 
= 
c\,x^{(\alpha_1)}x^{(\alpha_2-1)}\cdots x^{(\alpha_n-n+1)}
\]
where $c$ is some non zero constant.
\end{Lemma}

This justifies the introduction of $\tau ( \underline{\alpha} )$
in Subsection~{\ref{subsection-polynomial-space}}.

\begin{proof}
The matrix $\Delta_0(\underline{\alpha})$ satisfies the conditions of
Theorem~{\ref{minth}}.  Thus the lowest monomial is the product of
diagonal terms. The $j^{\text{th}}$ column is
$\mathcal{C}_{\alpha_j} (x)$, hence
its $j^{\text{th}}$ entry is:
\[
\binom{\alpha_j}{j-1}x^{(\alpha_j-j+1)}.
\]
Thus the diagonal monomial of
of $\det\, \Delta_0(\underline{\alpha})$ is:
\[
x^{(\alpha_1)}
x^{(\alpha_2 - 1)}
\cdots
x^{(\alpha_n - n + 1)}
=
x^{(\tau(\underline{\alpha}))}
\]
up to some nonzero constant.
\end{proof}

The weight of this diagonal monomial is:
\[
\alpha_1+1
+
\alpha_2
+\cdots+
\alpha_n-n+2
=
w_{\mathcal{B}} (\underline{\alpha}).
\]

\begin{Lemma}
All monomials of $\det\, \Delta_0(\underline{\alpha})$ have weight
$w_{\mathcal{B}}(\underline{\alpha})$:
\[
\det(\Delta_0(\underline{\alpha})) 
\in 
E_{w_{\mathcal{B}}(\underline{\alpha})}.
\]
\end{Lemma}

\begin{proof}
Consider the entry in the $i^{\text{th}}$ line and
$j^{\text{th}}$ column:
\[
m_{ij}
=
\binom{\alpha_j}{i-1}x^{(\alpha_j-i+1)}.
\]
We can write the determinant of $\Delta_0(\underline{\alpha})$ as:
\[
\det(\Delta_0(\underline{\alpha})) 
= 
\sum_{\sigma \in \mathfrak{S}_n}(-1)^{\epsilon(\sigma)}
\prod_{j=1}^n 
\binom{\alpha_j}{\sigma(j) - 1}
x^{(\alpha_j-\sigma (j)+1)}.
\]
For some fixed permutation $\sigma$, the sum gives us, up to some
multiplicative constant, the monomial:
\[
x^{(\alpha_1 - \sigma (1) + 1)}
x^{(\alpha_2 - \sigma (2) + 1)}
\cdots
x^{(\alpha_n - \sigma (n) + 1)},
\]
which has weight:
\begin{align*}
\sum_{k=1}^n 
\big(
\alpha_k - \sigma (k) + 2
\big) 
&= 
\sum_{k=1}^n 
\big(\alpha_k - k + 2\big)
+
\sum_{k=1}^n 
\big(k - \sigma (k)\big) \\
&=
w_{\mathcal{B}}(\underline{\alpha})
+
0.
\qedhere
\end{align*}
\end{proof}

We are now ready to prove our theorem.

In order to prove that the family 
$\big\{ \det\, \Delta_0 (\underline{ \alpha} ) \big\}_{
\underline{\alpha} \in \mathcal{B}}$ 
is a basis for $E$, we shall start by
proving it is free. To do that, 
with $s \geqslant 1$, 
let $\underline{\alpha}^1,
\dots, \underline{\alpha}^s$ be sequences in
$\mathcal{B}$, and let $\lambda_1, \dots , \lambda_s
\in \mathbb{C}^\ast$ be non-zero
coefficients. We must show that the linear
combination:
\begin{equation} \label{lincomb}
\lambda_1 \det\,\Delta_0(\underline{\alpha}^1)
+ \lambda_2 \det\,\Delta_0(\underline{\alpha}^2)
+ \dots 
+ \lambda_s \det\,\Delta_0(\underline{\alpha}^s)
\end{equation}
does not vanish. 

To see this, for $i = 1, \dots, s$, 
let us expand each determinant as a sum of its 
(nonzero) monomials:
\[
\aligned
\det(\Delta_0(\underline{\alpha}^i))
&
=
c_1^i x^{(\underline{a}_1^i)}
+ c_2^i x^{(\underline{a}_2^i)}
+ \dots
+ c_{k_i}^i x^{(\underline{a}_{k_i}^i)}
\\
&
=
c_i\,x^{(\tau(\underline{\alpha}^i))}
+\cdots.
\endaligned
\]
Now, amongst these monomials, there is the lowest one
with $c_i \neq 0$:
\[
c_i x^{(\tau(\underline{\alpha}^i))}
\eqno
{\scriptstyle{(1\,\leqslant\,i\,\leqslant\,s)}}.
\]
And in the midst of the smallest monomials, one might find the
smallest of the smallest:
\[
c_m x^{(\tau(\underline{\alpha}^m))},
\]
for some integer in $\llbracket 1,s \rrbracket$. 

The monomial 
$\lambda_m c_m x^{(\tau(\underline{\alpha}^m))}$
appears only once in the linear combination~{\ref{lincomb}},
so that~{\ref{lincomb}} does not vanish, indeed.

Lastly, we need to show that 
$\big\{ \det\, \Delta_0 (\underline{ \alpha} ) \big\}_{
\underline{\alpha} \in \mathcal{B}}$
is a generating family. Recall that we
have the direct sum decomposition:
\[ 
E 
=  
\bigoplus_{p \in \mathbb{N}} E_p. 
\]
Let $p \geqslant 0$ be some integer. We know that the family:
\[
\Big\{
\det(\Delta_0(\underline{\alpha}))
\Big\}
_{\underline{\alpha} \in \mathcal{B}_p}
\]
is a free family of $\text{Card}(\mathcal{B}_p)$ elements of
$E_p$. However, 
according to~{\ref{dim-Ep-Card-Ap-Card-Bp}},
the dimension of $E_p$ is precisely
$\text{Card}(\mathcal{B}_p)$, so that this free family is also a
generating family.

We deduce that the family generates every subspace $E_p$, thus it
generates the whole space $E$, and it is a basis.
\end{proof}

\begin{Example}
The natural basis of $E_5$ is:
\[
x'''' 
\ \ \ , \ \ \ 
xx''' 
\ \ \ , \ \ \ 
x'x'' 
\ \ \ , \ \ \ 
xxx'' 
\ \ \ , \ \ \ 
xx'x' 
\ \ \ , \ \ \ 
xxxx' 
\ \ \ \text{and} \ \ \ 
xxxxx,
\]
and the basis given by the determinants 
$\Delta_0(\underline{\alpha})$ is:
\begin{align*}
&\det \begin{pmatrix}
x''''
\end{pmatrix}
=
x''''
&\det \begin{pmatrix}
x & x'''' \\
0 & 4x'''
\end{pmatrix}
=
4xx''' \quad \quad \quad \quad \quad \ 
\\
&\det \begin{pmatrix}
x' & x''' \\
x & 3x''
\end{pmatrix}
=
3x'x'' - xx'''
&\det \begin{pmatrix}
x & x' & x'''' \\
0 & x & 4x'''\\
0 & 0 & 6x''
\end{pmatrix}
=
6xxx'' \quad \quad \ \ \
\\
&\det \begin{pmatrix}
x & x'' & x''' \\
0 & 2x' & 3x''\\
0 & x & 3x'
\end{pmatrix}
=
6xx'x' - 3xxx''
&\det \begin{pmatrix}
x & x' & x'' & x''''\\
0 & x & 2x' & 4x'''\\
0 & 0 & x & 6x'' \\
0 & 0 & 0 & 4x'
\end{pmatrix}
=
4xxxx'
\\
&\det \begin{pmatrix}
x & x' & x'' & x''' & x'''' \\
0 & x & 2x' & 3x'' & 4x''' \\
0 & 0 & x & 3x' & 6x'' \\
0 & 0 & 0 & x & 4x' \\
0 & 0 & 0 & 0 & x
\end{pmatrix}
=
xxxxx
\end{align*}
\end{Example}

%%%%%%%%%%%%%%%%%%%%%%%%%%%%%%%%%%%%%%%%%%%%%%%%%%%%%%%%%%%%%%%%%%%%%%
\subsection{The Two Changes of Affine Charts on $\mathbb{P}^1
(\mathbb{C})$}
%%%%%%%%%%%%%%%%%%%%%%%%%%%%%%%%%%%%%%%%%%%%%%%%%%%%%%%%%%%%%%%%%%%%%%

The affine coordinate of a point in the open $U_1 = \{X_1 \neq 0\}$ 
will be denoted
$y := \frac{X_0}{X_1}$. 
Hence, a local polynomial section of the Green-Griffiths bundle
can be represented as an element in:
\[
E_1 
= 
\mathbb{C}[y,y',y'',\dots].
\]
Recall the relations between the affine variables $x$ and $y$ and the
homogeneous variables $X_0$ and $X_1$:
\begin{align*}
x &= \frac{X_1}{X_0}, \\
y &= \frac{X_0}{X_1}, \\
x &= \frac{1}{y},
\end{align*}
from which one can deduce the relations for the derivatives:
\begin{align*}
x'  &= -\frac{y'}{y^2}, \\
x'' &= \frac{2y'y' - y''y}{y^3},
\end{align*}
and, more generally:
\[
x^{(n)} = \frac{\centersmallbullet}{y^{n+1}},
\]
where $\centersmallbullet$ designates some polynomial in $E_1$. 

Now, if $\underline{a}$ is some sequence in $\mathcal{A}$, then the
monomial $x^{(\underline{a})}$ can be written in the chart $U_1$ as:
\[
x^{(\underline{a})} 
= 
\frac{\centersmallbullet}{y^{w_{\mathcal{A}}(\underline{a})}}.
\]
The right-hand side has a pole at $y=0$, thus the local section
$x^{(\underline{a})}$ does not extend to a global holomorphic
section on $\mathbb{P}^1 (\mathbb{C})$. 
However, if $d \geqslant w_{\mathcal{A}}(\underline{a})$,
then the change of charts in the 
{\em twisted} bundle $E^{\text{GG}} \otimes 
\mathcal{O} (d) = E^{\text{GG}} (d)$ is
given by:
\[
x^{(\underline{a})} \longrightarrow y^{d}
\frac{\centersmallbullet}{y^{w_{\mathcal{A}}(\underline{a})}},
\]
the pole is therefore fully compensated. For some fixed integer $d \in
\mathbb{N}$, we know which monomials are in the space
${H}^0 (\mathbb{P}^1(\mathbb{C}),E_{k,
\infty}^{\text{GG}}(d))$ of global polynomial sections.

\begin{Fact}
For every $d \in \mathbb{N}$ and $k \geqslant d-1$,
\begin{equation}
\label{inclusion_naive}
\text{\rm Vect}
\big(
\{x^{(\underline{a})}\}
_{w_{\underline{\mathcal{A}}}(a) \leqslant d}
\big)
\,\subset\,
H^0
\big(
\mathbb{P}^1(\mathbb{C}),E^{\text{GG}}_{k, \infty}(d)
\big).
\end{equation}
\end{Fact}

Did we find all of them? The answer unfortunately is no. 
For example, consider the polynomial:
\[
2x'x' - xx''.
\]
Both of its monomials have weight $4$, so it is not in the space on
the left-hand side of~\eqref{inclusion_naive} for $d=3$. However, if
we apply the change of charts in $E_{k, \infty}^{\text{GG}}(3)$, we
see that the poles of order $4$ disappear:
\begin{align*}
2x'x' - xx'' \quad
&\longmapsto
\quad y^3 \Big(2\frac{(-y')}{y^2}\frac{(-y')}{y^2} 
- 
\frac{1}{y} \ \ \frac{2y'y' - yy''}{y^3}\Big) \\
&\ \ \ \ \ \ =
y^3\Big(\frac{yy''}{y^4}\Big) \\
&\ \ \ \ \ \ =
y''.
\end{align*}
Global sections of the twisted bundle $E^{\text{GG}}(d)$ are thus
sometimes linear combinations of monomials with weights larger than
$d$. In that case the poles of order greater than $d$ compensate each
other.

These linear combinations are in fact the ones of the determinants of
the $\det\, \Delta_0(\underline{\alpha})$ for $\underline{\alpha} \in
\mathcal{B}$. For instance:
\begin{align*}
\det\,\Delta_0((1,2))
&=
\begin{vmatrix}
x' & x'' \\
x  & 2x'
\end{vmatrix} \\
&=
2x'x' - xx''
\end{align*}
We will now discuss how the
$\det\, \Delta_0(\underline{\alpha})$ behave 
through changes of affine charts.

\begin{Definition}
For every $\underline{\alpha} \in \mathcal{B}$, let
$\mathcal{H}(\underline{\alpha})$ be the matrix:
\[
\mathcal{H}(\underline{\alpha})
=
\Big(
\mathcal{C}_{\alpha_1}(X_1) 
\ \Big\vert \
\mathcal{C}_{\alpha_2}(X_1) 
\ \Big\vert \ 
\cdots 
\ \Big\vert \
\mathcal{C}_{\alpha_n}(X_1)
\ \Big\vert \
\mathcal{C}_{n}(X_0)
\ \Big\vert \
\mathcal{C}_{n+1}(X_0)
\ \Big\vert \
\cdots
\ \Big\vert \
\mathcal{C}_{\alpha_n}(X_0)
\Big),
\]
where the columns have length $\alpha_n + 1$, so that the matrix
$\mathcal{H}(\underline{\alpha})$ is a square matrix.
\end{Definition}

We can dehomogenize this matrix in two ways:
\[
\frac{1}{X_0}\mathcal{H}(\underline{\alpha})
\ \ \ \text{and} \ \ \
\frac{1}{X_1}\mathcal{H}(\underline{\alpha}).
\]
The next theorem is a particular case of Theorem \ref{main3} and will be
useful to compute the determinant of $\frac{1}{X_i}
\mathcal{H}(\underline{\alpha})$.
It shows that these two dehomogenized determinants
depend only on $x$ and only on $y$.

\begin{Theorem} \label{algo}
For $i \in \{1, 0\}$, the determinant
\[
\det
\Big(
\frac{1}{X_i}
\mathcal{H}(\underline{\alpha})
\Big)
\]
can be computed, up to a $-1$ factor, by the following algorithm:

\begin{itemize}

\item In $\mathcal{H}(\underline{\alpha})$, remove every line and column containing $X_i^{(0)}$.

\item In the new matrix, replace $X_1^{(t)}$ by $x^{(t)}$ and $X_0^{(t)}$ by $y^{(t)}$ for all $t \in \mathbb{N}$. 

Denote these two new matrices by $\Delta_i(\underline{\alpha})$,
and more precisely: 
\[
\Delta_0(\underline{\alpha})(x),
\ \ \ \ \ \ \ \ \ \ \ \ \ \ \ \ \ \ \ \
\Delta_1(\underline{\alpha})(y).
\]
This
notation enters into conflict with the one in
definition~\ref{sectionslocales}, but the resulting matrices will in
fact coincide, so that there is no ambiguity.

\item compute the new determinant. 

\end{itemize}

\end{Theorem}

We prove this theorem in the general case 
in Section~\ref{changement_cartes}.

\begin{Example}
Put $\underline{\alpha} = (1,2)$. 
The homogeneous matrix associated to $\underline{\alpha}$ is:
\[
\mathcal{H}(\underline{\alpha})
=
\begin{pmatrix}
X_1' & X_1'' & X_0'' \\
X_1  & 2X_1' & 2X_0' \\
0 & X_1 & X_0
\end{pmatrix}.
\]
Thus, by the theorem:
\begin{align*}
\det
\Big(
\frac{1}{X_0}\mathcal{H}(\underline{\alpha})
\Big) 
&=
\begin{vmatrix} 
x' & x'' \\
x  & 2x'
\end{vmatrix}
\\
&=
2x'x' - xx'',
\end{align*}
and
\begin{align*}
\det
\Big(
\frac{1}{X_1}\mathcal{H}(\underline{\alpha})
\Big) 
&=
\begin{vmatrix} 
y''
\end{vmatrix}
\\
&=
y''.
\end{align*}
\end{Example}

For some element $\underline{\alpha}$ in $\mathcal{B}$,
we can now write $\Delta_0(\underline{\alpha})$ in the chart $U_1$:
\begin{align*}
\det\big(\Delta_0(\underline{\alpha})(x)\big)
&= 
\pm \det
\Big(
\frac{1}{X_0}
\mathcal{H}(\underline{\alpha})
\Big) 
\\
&=
\pm 
\Big(
\frac{X_1}{X_0}
\Big)^{\alpha_n + 1}
\det
\Big(
\frac{1}{X_1}
\mathcal{H}(\underline{\alpha})
\Big) 
\\
[\text{\rm Theorem~{\ref{algo}}}]
\ \ \ \ \ \ \ \ \ \ \ \ \ \ \ \ \ \ \ \ \ \ \ \ \ \
&=
\pm
\frac{\det\big(\Delta_1(\underline{\alpha})(y)\big)}{y^{\alpha_n + 1}}.
\end{align*}

From this change of charts formula one deduces the

\begin{Corollary}
If $\underline{\alpha} = (\alpha_1, \dots ,\alpha_n)$ is a sequence of
$\mathcal{B}$ such that $\alpha_n < d$, then
$\det(\Delta_0(\underline{\alpha}))$ is a global section of the
twisted bundle $E_{k,\infty}^{\text{GG}}(d)$.
\end{Corollary}

\begin{Notation}
Denote the subset of sequences $(\alpha_1, \alpha_2,
\dots , \alpha_n) \in \mathcal{B}$ 
such that $\alpha_n < d$ by $\mathcal{B}^d$.
\end{Notation}

\begin{Corollary}
\label{coco2}
For every $d \in \mathbb{N}$, and for every $k \geqslant d-1$, 
one has:
\[
\Big\{
\det(\Delta_0(\underline{\alpha}))
\Big\}
_{\underline{\alpha} \in \mathcal{B}^d}
\subset
H^0
\Big(
\mathbb{P}^1(\mathbb{C}),
E^{\text{GG}}_{k, \infty}(d)
\Big),
\]
so that one has the following lower bound:
\[
\dim
\Big(
H^0
\Big(
\mathbb{P}^1(\mathbb{C}),
E^{\text{GG}}_{k, \infty}(d)
\Big)
\Big)
\geqslant
2^d.
\]
\end{Corollary}

%%%%%%%%%%%%%%%%%%%%%%%%%%%%%%%%%%%%%%%%%%%%%%%%%%%%%%%%%%%%%%%%%%%%%%
\subsection{The non-vanishing poles}\label{non_vanishing_pole}
%%%%%%%%%%%%%%%%%%%%%%%%%%%%%%%%%%%%%%%%%%%%%%%%%%%%%%%%%%%%%%%%%%%%%%

The last step is to show that the lower bound of Corollary~\ref{coco2}
is an equality. Recall that a local section 
of the Green-Griffiths bundle
on the open $U_1$ is given
by a polynomial in the space:
\[
E_1 
= 
\mathbb{C}[y,y',y'', \dots].
\]
To describe monomials in this space, we use the notation:
\[
y^{(\underline{a})} 
=
y^{(a_1)}y^{(a_2)} \cdots y^{(a_n)}
\]
for every sequence $\underline{a} = (a_1, \dots , a_n)$ in
$\mathcal{A}$.

We shall order the monomials of this space in the same way we ordered
the monomials in $E = \mathbb{C}[x,x',x'', \dots]$:
\begin{align*}
1 \prec y& \prec yy \prec yyy \prec \cdots \\
\dots \prec y'& \prec y'y \prec y'yy \prec \cdots \\
\dots \prec y'y'& \prec y'y'y \prec y'y'yy \prec \cdots \\
 & \dots  \\
 & \dots  \\
\dots \prec y''& \prec y''y \prec y''yy \prec \cdots \\
\dots \prec y''y'& \prec y''y'y \prec y''y'yy \prec \dots.
\end{align*}

Let $\underline{\alpha}$ be an element in $\mathcal{B}$. The
determinant of $\Delta_0(\underline{\alpha})$ is a local section of
the Green-Griffiths bundle on the open $U_0$. The change of charts is
given by:
\[
\det (
\Delta_0(\underline{\alpha})
)
=
\frac{1}{y^{\alpha_n + 1}}\det (
\Delta_1(\underline{\alpha})
).
\]
If we develop the determinant on the right-hand side, we get a sum of
rational terms. Each one of those terms can be written in a unique way as a
reduced fraction:
\[
\frac{y^{(\underline{a})}}{y^s},
\]
where $s \geqslant 0$ and $\underline{a} \in \mathcal{A}$ is a
sequence such that $a_1 \geqslant 1$.

On those reduced fractions, we will define an order. If $s, s'$ are
two integers, and $\underline{a}, \underline{b}$ are in $\mathcal{A}$
with $a_1 \geqslant 1$, $b_1 \geqslant 1$, then we say that:
\[
\frac{y^{(\underline{a})}}{y^s}
\prec
\frac{y^{(\underline{b})}}{y^{s'}}
\]
if and only if:
\[
s > s'
\]
or
\[
s = s' 
\ \ \ \text{and} \ \ \
y^{(\underline{a})} \prec y^{(\underline{b})}.
\]

Our goal is now to find the lowest rational term in:
\[
\frac{1}{y^{\alpha_n + 1}}
\Delta_1(\underline{\alpha}).
\]
By definition, $\Delta_1(\underline{\alpha})$ is the submatrix of
$\mathcal{H}(\underline{\alpha})$ where all lines and columns
containing $X_1$ have been removed, and where $X_0$ has been replaced
by $y$.

We have thus removed the $n$ first columns, and the lines $\alpha_1 +
1, \alpha_2 + 1, \dots, \alpha_N + 1$.  Let $L_t$ be the $t$-th line
of $\Delta_1(\underline{\alpha})$. We may write $L_t$ for the first
values of $t$:
\begin{align*}
L_1  &= \big(\quad \quad \quad y^{(n)} 
\quad \quad \quad , \quad \quad \quad 
y^{(n + 1)} 
\quad \quad \quad , \quad \dots \quad , 
\quad y^{(\alpha_n)}   
\big), \\
L_2  &= \big(\quad \quad \quad ny^{(n-1)} 
\quad \quad , \quad \quad 
(n+1)y^{(n)} 
\quad \ , \quad \dots \quad ,
\alpha_ny^{(\alpha_n - 1)} 
\big), \\
&\,\,\:\vdots \\
L_{\alpha_1} &= \Big(
{\textstyle{\binom{n}{\alpha_1-1}}}
y^{(n-\alpha_1+1)},
{\textstyle{\binom{n+1}{\alpha_1-1}}}
y^{(n-\alpha_1+2)},
\dots,
{\textstyle{\binom{\alpha_n}{\alpha_1-1}}}
y^{(\alpha_n-\alpha_1+1)} \Big), \\
L_{\alpha_1 + 1} &= \Big(
{\textstyle{\binom{n}{\alpha_1+1}}}
y^{(n-\alpha_1-1)},
{\textstyle{\binom{n+1}{\alpha_1+1}}}
y^{(n-\alpha_1)},
\dots,
{\textstyle{\binom{\alpha_n}{\alpha_1+1}}}
y^{(\alpha_n-\alpha_1-1)} \Big).
\end{align*}
Let us introduce some notations. We will say that a line is of type
$\alpha$ if it has the form:
\[
\big(
c_{\alpha} y^{(\alpha)}, c_{\alpha + 1} y^{(\alpha + 1)}, 
\dots
\big),
\]
where the complex coefficients $c_{r}$ are zero if and only if $r <
0$. For every integer $t$, we can give the type of line $L_t$:

\begin{itemize}

\item If $1 \leqslant t \leqslant \alpha_1$, 
\quad \quad \quad \ \ \ then $L_t$ is of type $n-t+1$. 

\item If $\alpha_1 + 1 \leqslant t \leqslant \alpha_2 - 1$, then $L_t$ is of type $n-t$.

\item If $\alpha_2 \leqslant t \leqslant \alpha_3 - 2$, 
\quad \quad then $L_t$ is of type $n-t-1$.

\item $\cdots$

\item If $\alpha_{n-1} -n+3 \leqslant t \leqslant \alpha_n - n+1$, then $L_t$ is of type $-t+2$. 
\end{itemize}

By definition, the $k$-th entry of a line of type $\alpha$ is, up to
some multiplicative constant, $y^{(\alpha + k - 1)}$. Now that we have
the types of the lines of $\Delta_1(\underline{\alpha})$, we can take
a look at the diagonal terms, which are the $t$-th entry on the $t$-th
line, for every $t \leqslant \alpha_n - n + 1$.

\begin{itemize}

\item If $1 \leqslant t \leqslant \alpha_1$, the $t$-th entry of the $t$-th line is $y^{(n)}$. 

\item If $\alpha_1 + 1 \leqslant t \leqslant \alpha_2 - 1$, the $t$-th entry of the $t$-th line is $y^{(n-1)}$. 

\item If $\alpha_2 \leqslant t \leqslant \alpha_3 - 2$, the $t$-th entry of the $t$-th line is $y^{(n-2)}$. 

\item $\cdots$ 

\item If $\alpha_{n-1} - n + 3 \leqslant t \leqslant \alpha_n - n + 1$, the $t$-th entry of the $t$-th line is $y^{(1)}$.

\end{itemize}

We have thus acces to the diagonal term of the determinant
$\det\, \Delta_1(\underline{\alpha})$.

\begin{Fact}
The product of the diagonal terms in $\Delta_1(\underline{\alpha})$
is, up to some multiplicative non-zero scalar:
\[
{(y')}^{\alpha_n - \alpha_{n-1} - 1}{(y'')}^{\alpha_{n-1} 
- \alpha_{n-2} - 1}
\dots
{(y^{(n-1)})}^{\alpha_2 - \alpha_1 - 1}
{(y^{(n)})}^{\alpha_1}.
\]
\end{Fact}

Note that this term does not have $y$ as a factor. Thus the fraction:
\[
\frac{
{(y')}^{\alpha_n - \alpha_{n-1} - 1}{(y'')}^{
\alpha_{n-1} - \alpha_{n-2} - 1}
\dots
{(y^{(n-1)})}^{\alpha_2 - \alpha_1 - 1}
{(y^{(n)})}^{\alpha_1}
}
{y^{\alpha_n + 1}}
\]
is already in its reduced form. The power of the denominator is
$\alpha_n + 1$, which is the maximum we can get when developing:
\[
\frac{\Delta_1(\underline{\alpha})}{y^{\alpha_n + 1}}.
\]
Furthermore, $\Delta_1(\underline{\alpha})$ satisfies the conditions
of~\eqref{minth}, so that its minimal monomial is obtained on the
diagonal. From those two facts one deduces the

\begin{Property}
The minimal rational term in:
\[
\frac{\Delta_1(\underline{\alpha})}{y^{\alpha_n + 1}}
\]
is:
\[
\frac{
{(y')}^{\alpha_n - \alpha_{n-1} - 1}{(y'')}^{\alpha_{n-1} - \alpha_{n-2} - 1}
\dots
{(y^{(n-1)})}^{\alpha_2 - \alpha_1 - 1}
{(y^{(n)})}^{\alpha_1}
}
{y^{\alpha_n + 1}}.
\]
\end{Property}

In the natural basis $\{x^{(\underline{a})}\}_{\underline{a} \in \mathcal{A}}$ for $E$, one can consider the following two elements
\[
x'x' \quad \quad \text{and} \quad \quad xx'',
\]
which, in the chart $U_1$, can be written as
\[
\frac{y'y'}{y^4} \quad \quad \text {and} \quad \quad \frac{-y''}{y^3} + \frac{2y'y'}{y^4},
\]
so that they share the same minimal term. In the next lemma we prove that this does not happen in the basis $\{\Delta_0(\underline{\alpha})\}_{\underline{\alpha} \in \mathcal{B}}$.
\begin{Lemma}
If $\underline{\alpha} \neq \underline{\beta}$ 
are two distinct sequences in $\mathcal{B}$, then:
\[
\frac{\Delta_1(\underline{\alpha})}{y^{\alpha_n + 1}}
\ \ \ \text{and} \ \ \
\frac{\Delta_1(\underline{\beta})}{y^{\beta_m + 1}}
\]
have different smallest terms.
\end{Lemma}

\begin{proof}
Let $\underline{\alpha} = (\alpha_1, \dots, \alpha_n)$ and
$\underline{\beta} = (\beta_1, \dots, \beta_m)$ be two sequences in
$\mathcal{B}$ such that
$\frac{\Delta_1(\underline{\alpha})}{y^{\alpha_n + 1}}$ and
$\frac{\Delta_1(\underline{\beta})}{y^{\beta_m + 1}}$ have the same
minimal term.

In that case, the powers of $y$ at the denominators must coincide, so
that $\alpha_n = \beta_m$. Furthermore, the number of factors in the
terms of $\Delta_1(\underline{\alpha})$ is equal to $\alpha_n + 1 -
n$. Since this must be equal to the number of factors of the terms of
$\Delta_1(\underline{\beta})$, we deduce that $n = m$: the two
sequences $\underline{\alpha}$ and $\underline{\beta}$ have same
length.

Now, by looking at the number of factors that are $y^{(l)}$ for all $l
\in \llbracket 1, n \rrbracket$ in $\Delta_1(\underline{\alpha})$, and
by comparing these numbers with the ones of
$\Delta_1(\underline{\beta})$, we deduce that:
\[
\aligned
\alpha_n - \alpha_{n-1} - 1 &= \beta_n - \beta_{n-1} - 1, \\
\alpha_{n-1} - \alpha_{n-2} - 1 &= \beta_{n-1} - \beta_{n-1} - 1, \\
&\vdots \\
\alpha_2 - \alpha_1 - 1 &= \beta_2 - \beta_1 - 1, \\
\alpha_1 &= \beta_1. \\
\endaligned
\]

By induction, we see that $\alpha_{n - t} = \beta_{n - t}$ for all
$t \in \llbracket 0, n - 1 \rrbracket$, which concludes the proof by
contraposition.
\end{proof}

We are now ready to prove that the inclusion in Corollary~\ref{coco2}
is, in fact, an equality.

\begin{Theorem}
For every integer $d$ and $k \geqslant d-1$, the family
$\{\Delta_0(\underline{\alpha})\}_{\underline{\alpha} \in
\mathcal{B}^d}$ is a basis of:
\[
H^0(\mathbb{P}^1(\mathbb{C}), E^{\text{GG}}_{k, \infty}(d)).
\]
Thus:
\[
\dim\,
H^0\big(\mathbb{P}^1(\mathbb{C}),\,
E^{\text{GG}}_{k, \infty}(d)
\big) 
= 
2^d.
\]
\end{Theorem}

\begin{proof}
Let $P$ be a holomorphic section in $E^{\text{GG}}(d)$, that we may
write in the basis $\big\{\det\,
\Delta_0( \underline{\alpha}) 
\big\}_
{\underline{\alpha} 
\in
\mathcal{B}}$ as:
\begin{equation}
\label{combinaisonlin}
P 
= 
\sum_{i = 1}^k \lambda_i\,
\det\,\Delta_0(\underline{\alpha}^i).
\end{equation}
After the change of chart, each determinant $\det\,
\Delta_0(\underline{\alpha}^i)$ will give us a smallest term:
\[
\frac{y^{(\underline{a}^i)}}{y^{\alpha_{n_i}^i + 1}}.
\]
And amongst these terms, one can find the smallest of the smallest:
\begin{equation}
\label{minimum}
\frac{y^{(\underline{a}^m)}}{y^{\alpha_{n_m}^m + 1}} 
\ \ \ \text{for some} \ 
m \in \llbracket1,k\rrbracket.
\end{equation}
In the sum~\eqref{combinaisonlin}, this term does not vanish. 

Since $P$ is a holomorphic section of $E^{\text{GG}}(d)$, the pole of
this smallest term must be compensated by $y^d$.

Hence:
\[
\alpha_{n_m}^m + 1 \leqslant d,
\]
and thus we have, for all $i \in \llbracket1,k\rrbracket$:
\[
\alpha_{n_i}^i + 1
\leqslant
\alpha_{n_m}^m + 1
\leqslant
d,
\]
as~\eqref{minimum} is the smallest term, so it has the biggest power
in its denominator.

We conclude that $\underline{\alpha}^i$ is in $\mathcal{B}^d$ for all
$i \in \llbracket1,k\rrbracket$.
\end{proof}

%%%%%%%%%%%%%%%%%%%%%%%%%%%%%%%%%%%%%%%%%%%%%%%%%%%%%%%%%%%%%%%%%%%%%%
\section{The general case}\label{cas_general}
%%%%%%%%%%%%%%%%%%%%%%%%%%%%%%%%%%%%%%%%%%%%%%%%%%%%%%%%%%%%%%%%%%%%%%

In the general case, a point $p$ in the projective space
$\mathbb{P}^N(\mathbb{C})$ has $N + 1$ homogeneous coordinates 
$[X_0 \colon X_1\colon \cdots \colon X_N]$ 
and $N$ affine coordinates $x_{01}, x_{02}, \dots,
x_{0N}$ in the chart $U_0$. We will follow the same strategy than in
the one-dimensional case: we start by introducing a basis of the
polynomial space given by determinants, then we describe the changes
of charts. Finally, in the chart $U_1$, we will define an order on the
rational terms, and prove that every element of our basis yields a
different smallest term.

%%%%%%%%%%%%%%%%%%%%%%%%%%%%%%%%%%%%%%%%%%%%%%%%%%%%%%%%%%%%%%%%%%%%%%
\subsection{The polynomial space of local sections}
%%%%%%%%%%%%%%%%%%%%%%%%%%%%%%%%%%%%%%%%%%%%%%%%%%%%%%%%%%%%%%%%%%%%%%

A local, polynomial section of the Green -Griffiths bundle over the
open $U_0$ is given by a polynomial in the space:
\[
E
=
\mathbb{C}[x_{01}, \dots, x_{0N}, x_{01}', \dots, x_{0N}', \dots].
\]
The natural basis of $E$ is given by the monomials. For every sequence
$\underline{a} = (a_1, \dots, a_n)$ in $\mathcal{A}$, and for
every integer $i \in \llbracket 1, N \rrbracket$, we use the notation:
\[
x_{0i}^{(\underline{a})}
=
x_{0i}^{(a_1)}\cdots x_{0i}^{(a_n)},
\]
and, for a $N$-tuple $\underline{\underline{a}} = (\underline{a^1},
\underline{a^2}, \dots, \underline{a^N})$ in $\mathcal{A}^N$, we
define the monomial:
\[
x_0^{\left(\underline{\underline{a}}\right)}
=
x_{01}^{(\underline{a^1})}x_{0 2}^{(\underline{a^2})} 
\cdots x_{0N}^{(\underline{a^N})}.
\]

\begin{Fact}
The family
\[
\Big\{ 
x_0^{\left( \underline{\underline{a}}\right)}
\Big\}_{\underline{\underline{a}} \in \mathcal{A}^N}
\]
is a basis of $E$.
\end{Fact}

We define $\mathcal{B}^{N+}$ to be the subset of $\mathcal{B}^N$ of
$N$-tuples $(\underline{\alpha^1}, \underline{\alpha^2}, \dots,
\underline{\alpha^N}) \in \mathcal{B}^{N}$ such that:
\[
\aligned
0 \quad \quad \quad &\leqslant \alpha_1^N < \alpha_2^N < \cdots < \alpha_{n_N}^N, \\
n_N \quad \quad \ \ &\leqslant \alpha_1^{N - 1} < \alpha_2^{N - 1} < \cdots < \alpha_{n_{N - 1}}^{N - 1}, \\
n_N + n_{N - 1} \quad &\leqslant \alpha_1^{N - 2} < \alpha_2^{N - 2} < \cdots < \alpha_{n_{N - 2}}^{N - 2}, \\
&\,\,\:\vdots \\
n_N + n_{N - 1} + \cdots + n_2 &\leqslant \alpha_1^1 < \alpha_2^1 < \cdots < \alpha_{n_1}^1.
\endaligned
\]
Throughout this section, we will often use the abbreviations:
\[
\aligned
s_{N+1} &= 0, \\
s_N &= n_N,\\
s_{N - 1} &= n_N + n_{N - 1}, \\
&\,\,\:\vdots \\
s_1 &= n_N + n_{N - 1} + \cdots + n_1.
\endaligned
\]

There is a natural correspondence between 
$\mathcal{A}^N$ and $\mathcal{B}^{N+}$ given by:
\[
\begin{array}{ccccl}
\tau^+ & : & \mathcal{B}^{N+} & \to & \quad \quad \quad \quad \mathcal{A}^N 
\\
& & 
(\alpha^j_i)_{\substack{1 \leqslant j \leqslant N \\ 1 \leqslant i \leqslant n_j}}
& 
\mapsto & 
(\alpha^j_i - s_{j+1} - i + 1)_{\substack{1 \leqslant j \leqslant N \\ 1 \leqslant i \leqslant n_j}}
\\
\end{array}.
\]

\begin{Lemma}
The map $\tau^+$ is a bijection.
\end{Lemma}

\begin{proof}
The inverse map is given explicitly by:
\[
\begin{array}{ccccl}
(\tau^+)^{-1} & : & \mathcal{A}^{N} & \to & \quad \quad \quad \quad \mathcal{B}^{N+} 
\\
& & 
(a^j_i)_{\substack{1 \leqslant j \leqslant N \\ 1 \leqslant i \leqslant n_j}}
& 
\mapsto & 
(a^j_i + s_{j+1} + i - 1)_{\substack{1 \leqslant j \leqslant N \\ 1 \leqslant i \leqslant n_j}}
\end{array}.
\qedhere
\]
\end{proof}

%%%%%%%%%%%%%%%%%%%%%%%%%%%%%%%%%%%%%%%%%%%%%%%%%%%%%%%%%%%%%%%%%%%%%%
\subsection{Weighting and ordering the monomials}
%%%%%%%%%%%%%%%%%%%%%%%%%%%%%%%%%%%%%%%%%%%%%%%%%%%%%%%%%%%%%%%%%%%%%%

Suppose we want to write $x_{0i}^{(a)}$ in the chart $U_j$, for some $j \in \llbracket 1, N \rrbracket$. 
The changes of charts are given by:
\[
\aligned
x_{0i}^{(a)} &= (\frac{x_{ji}}{x_{j0}})^{(a)} 
\ \ \ \text{if} \ i \neq j, \\
x_{0i}^{(a)} &= (\frac{1}{x_{ii}})^{(a)} 
\ \ \ \text{if} \ i = j.
\endaligned
\]

In both cases, we can write:
\[
x_{0i}^{(a)}
=
\frac{\centersmallbullet}{\chi^{a+1}},
\]
where $\centersmallbullet$ is some polynomial in $E$, and $\chi$ is some affine
variable.

We will therefore say that, for all $i \in \llbracket 1, N \rrbracket$
and $a \in \mathbb{N}$, the variable $x_{0i}^{(a)}$ has weight
$a+1$. The weight of a monomial will naturally be defined to be the
sum of the weights of the variables which constitute the
monomial. Formally, we define a weight function on the set
$\mathcal{A}^N$:
\[
\begin{array}{ccccl}
w_{\mathcal{A}^N} & : & \mathcal{A}^N & \to & 
\quad \quad \quad \quad \quad 
\mathbb{N} 
\\
& & (\underline{a^j})_{j \in \llbracket 1, N \rrbracket} & \mapsto & w_{\mathcal{A}}(\underline{a^1}) + w_{\mathcal{A}}(\underline{a^2}) + \cdots + w_{\mathcal{A}}(\underline{a^N})
\\
\end{array},
\]
where $w_{\mathcal{A}}$ is the weight function on the set $\mathcal{A}$ defined in Section~\ref{poids_monomes}.

We can use the map $\tau^+$ to pull back this weight function on the set $\mathcal{B}^{N+}$:
\[
\begin{array}{ccccl}
w_{\mathcal{B}^{N+}} & : & \mathcal{B}^{N+} & \to & 
\quad \
\mathbb{N} 
\\
& & \underline{\underline{\alpha}} & \mapsto & w_{\mathcal{A}^N}(\tau^+(\underline{\underline{\alpha}}))
\\
\end{array}.
\]

We can now, for every integer $p$, define the two finite sets:
\[
\mathcal{A}^N_p 
=
\Big\{\underline{\underline{a}} \in \mathcal{A}^N 
\colon\,\,
w_{\mathcal{A}^N}(\underline{\underline{a}}) = p
\Big\},
\]
and:
\[
\mathcal{B}^{N+}_p 
=
\Big\{\underline{\underline{\alpha}} \in \mathcal{B}^{N+} 
\colon\,\, 
w_{\mathcal{B}^{N+}}(\underline{\underline{\alpha}}) = p
\Big\},
\]
which have same cardinality, for $\tau^+$ induces a bijection between
them.

Furthermore, we will say that the monomial
$x_0^{(\underline{\underline{a}})}$ has weight
$w_{\mathcal{A}^N}(\underline{\underline{a}})$, and we will define the
space of polynomials of weight $p$:
\[
E_p
=
\text{Vect}\Big(x_0^{
\left(\underline{\underline{a}}\right)}
\Big)_{\underline{\underline{a}} \in \mathcal{A}_p^N}.
\]

A basis of $E_p$ is given by the monomials associated to sequences in
$\mathcal{A}_p^N$, so that we have, for every $p \in \mathbb{N}$, the
triple equality:
\[
\dim (E_p)
=
\text{Card}(\mathcal{A}_p^N)
=
\text{Card}(\mathcal{B}_p^{N+}).
\]

In Section~\ref{ordre_monomes}, we ordered the monomials in the
variables $x, x', x'', \dots$. We shall generalize this order to
include all the affine coordinates $x_{01}, x_{02}, \dots, x_{0N}$. We
start by ordering the variables:
\[
\aligned
1 &\prec x_{0N} \prec x_{0N}' \prec x_{0N}'' \prec \cdots \\
&< x_{0 N-1} < x_{0 N-1}' < x_{0 N-1}'' < \cdots \\
&\vdots \\
&< x_{01} < x_{01}' < x_{01}'' < \cdots.
\endaligned
\]

Now, suppose we have two monomials $\chi_1 \chi_2 \cdots \chi_n$ and
$\chi_1'\chi_2'\cdots \chi_{m}'$ that we want to compare. We may also
assume that their variables are ordered:
\[
\chi_1 \leqslant \chi_2 \leqslant \cdots \leqslant \chi_n
\ \ \ \text{and} \ \ \
\chi_1' \leqslant \chi_2' \leqslant \cdots \leqslant \chi_m'.
\]

We will say that $\chi_1 \chi_2 \cdots \chi_n \prec
\chi_1'\chi_2'\cdots \chi_{m}'$ if and only if:
\[
\chi_n \prec \chi_m',
\]
or
\[
\chi_n = \chi_m'
\ \ \text{and} \ 
\chi_1 \chi_2 \cdots \chi_{n-1} \prec \chi_1'\chi_2'\cdots \chi_{m-1}'.
\]
By induction, this defines a total ordering on the set of monomials.

%%%%%%%%%%%%%%%%%%%%%%%%%%%%%%%%%%%%%%%%%%%%%%%%%%%%%%%%%%%%%%%%%%%%%%
\subsection{The determinant sections}
%%%%%%%%%%%%%%%%%%%%%%%%%%%%%%%%%%%%%%%%%%%%%%%%%%%%%%%%%%%%%%%%%%%%%%

Following what we did in the case $N=1$, we will now introduce a new
basis of the polynomial space $E$ in which the elements are given by
determinants. The matrices we will have to consider must contain all
the variables $x_{01}, x_{02}, \dots, x_{0N}$.

\begin{Definition}
Let $\underline{\underline{\alpha}}$ be a $N$-tuple of sequences in
$\mathcal{B}^{N+}$, that we may write as:
\[
\underline{\underline{\alpha}}
=
((\alpha^1_1, \alpha^1_2, \dots, \alpha^1_{n_1}), \dots , (\alpha^N_1, \alpha^N_2, \dots, \alpha^N_{n_N})).
\]
To $\underline{\underline{\alpha}}$, we associate the affine matrix
$\Delta_0(\underline{\underline{\alpha}})$ with columns:
\[
\mathcal{C}_{\alpha^N_1}(x_{0N}), \
\mathcal{C}_{\alpha^N_2}(x_{0N}), \
\dots, \
\mathcal{C}_{\alpha^N_{n_N}}(x_{0N}), \
\mathcal{C}_{\alpha^{N - 1}_1}(x_{0 N-1}), \
\dots, \
\mathcal{C}_{\alpha^1_{n_1}}(x_{0 1}).
\]
The length of the columns is chosen to be such that the resulting
matrix is a square, so they are of length $s_N = n_1 + n_2 + \cdots +
n_N$.
\end{Definition}

We can now state the main result of this section.

\begin{Theorem}
\label{th_base2}
The family:
\[
\Big\{ 
\det
\Delta_0(\underline{\underline{\alpha}})
\Big\}_{\underline{\underline{\alpha}} \in \mathcal{B}^{N+}}
\]
is a basis of the polynomial space $E$.
\end{Theorem}

We will use lemmas analogous to the ones we used to prove 
Theorem~\ref{th_base}.

\begin{Lemma}
When expanding the determinant 
$\det \Delta_0(\underline{\underline{\alpha}})$ 
as a linear combination of monomials in $E$, 
the smallest monomial is:
\[
x^{\left({\tau^+(\underline{\underline{\alpha}})}\right)}.
\]
\end{Lemma}

\begin{proof}
Consider the column number $s_{j-1} + i$, with $i \in \llbracket 1,
n_j \rrbracket$. Its diagonal entry is on line $s_{j - 1} +
i$. Since the column is $\mathcal{C}_{\alpha^{j-1}_i}(x_j)$, this
corresponds to $x_{0j}^{(\alpha^j_i - s_{j - 1} - i + 1)}$, up to some
nonzero binomial coefficient. The product of those variables is
$x^{(\tau^+(\underline{\underline{\alpha}}))}$. Since the matrix
$\Delta_0(\underline{\underline{\alpha}})$ satisfies the conditions of
Theorem~\ref{minth}, it is the smallest monomial.
\end{proof}

\begin{Lemma}
For every $\underline{\underline{\alpha}}$ in $\mathcal{B}^{N+}$, the
monomials of $\det \Delta_0( \underline{ \underline{\alpha}})$ all
have weight $w_{\mathcal{B}^{N+}}( \underline{ \underline{\alpha}})$:
\[
\det \Delta_0(\underline{\underline{\alpha}})
\in
E_{w_{\mathcal{B}^{N+}}(\underline{\underline{\alpha}})}.
\]
\end{Lemma}

\begin{proof}
The determinant can be written as a sum over the permutations:
\[
\det \Delta_0(\underline{\underline{\alpha}})
=
\sum_{\sigma \in \mathfrak{S}_{s_1}}
(-1)^{\epsilon (\sigma)}
\prod_{j=1}^N
\prod_{i = 1}^{n_j}
\binom{\alpha^j_i}{\sigma (s_{j+1} + i) - 1}
x_{0j}^{\alpha^j_i - \sigma (s_{j+1} + i) + 1}.
\]
In this sum, every monomial has weight:
\[
\sum_{j=1}^N
\sum_{i = 1}^{n_j}
\alpha^j_i - \sigma (s_{j+1} + i) + 2
=
\sum_{j=1}^N
\sum_{i = 1}^{n_j}
\alpha^j_i - s_{j+1} - i + 2.
\]
Which is, by definition, the weight
$w_{\mathcal{B}^{N+}}(\underline{\underline{\alpha}})$.
\end{proof}

We can now prove theorem \eqref{th_base2}.

\begin{proof}
The first Lemma implies that the family of determinants $\det
\Delta_0(\underline{\underline{\alpha}})$ is free, as any linear
combination:
\[
\lambda_1 \det
\Delta_0(\underline{\underline{\alpha}}^1)
+
\lambda_2 \det
\Delta_0(\underline{\underline{\alpha}}^2)
+
\cdots
+
\lambda_m \det
\Delta_0(\underline{\underline{\alpha}}^m)
\]
will have a smallest term that does not vanish.
Hence, for every weight $p$, the family:
\[
\Big\{ 
\det
\Delta_0(\underline{\underline{\alpha}})
\Big\}_{\underline{\underline{\alpha}} \in \mathcal{B}^{N+}_p}
\]
is a free family of elements in $E_p$. Since the dimension of $E_p$
corresponds to the number of elements in this family, the family is in
fact a basis. We deduce that the determinants $\det
\Delta_0(\underline{\underline{\alpha}})$ generate the spaces $E_p$,
so they generate the whole space:
\[
E = \bigoplus_{p \in \mathbb{N}} E_p.
\qedhere
\]
\end{proof}

%%%%%%%%%%%%%%%%%%%%%%%%%%%%%%%%%%%%%%%%%%%%%%%%%%%%%%%%%%%%%%%%%%%%%%
\subsection{The change of charts}\label{changement_cartes}
%%%%%%%%%%%%%%%%%%%%%%%%%%%%%%%%%%%%%%%%%%%%%%%%%%%%%%%%%%%%%%%%%%%%%%

For some element $\underline{\underline{\alpha}}$ in
$\mathcal{B}^{N+}$, we introduce the homogeneous matrix
$\mathcal{H}(\underline{\underline{\alpha}})$, defined to be the
matrix with the columns:
\[
\Big(
\mathcal{C}_{\alpha^N_1}(X_N) 
\ \Big\vert \
\mathcal{C}_{\alpha^N_2}(X_N)
\ \Big\vert \
\dots
\ \Big\vert \
\mathcal{C}_{\alpha^N_{n_N}}(X_N)
\ \Big\vert \
\mathcal{C}_{\alpha^{N - 1}_1}(X_{N-1})
\ \Big\vert \
\mathcal{C}_{\alpha^{N - 1}_2}(X_{N-1})
\ \Big\vert \
\dots
\ \Big\vert \
\mathcal{C}_{\alpha^1_{n_1}}(X_1)
\ \Big\vert \ ,
\]
followed by the columns in the variable $X_0$:
\[
\ \Big\vert \
\mathcal{C}_{s_1}(X_0)
\ \Big\vert \
\mathcal{C}_{s_1 + 1}(X_0)
\ \Big\vert \
\cdots
\ \Big\vert \
\mathcal{C}_{M(\underline{\underline{\alpha}})}(X_0)
\Big),
\]
where $M(\underline{\underline{\alpha}})$ is the maximum of the set
$\{ \alpha^j_i \}_{\substack{1 \leqslant j \leqslant N \\ 1 \leqslant
i \leqslant n_j}}$. The length of the columns are chosen to be
$M(\underline{\underline{\alpha}}) + 1$, so that
$\mathcal{H}(\underline{\underline{\alpha}})$ is a square matrix.

Note that the columns of $\mathcal{H}(\underline{\underline{\alpha}})$
containing $X_0^{(0)}$ are the last $M(\underline{\underline{\alpha}})
- s_1 + 1$ columns, and similarly the lines of
$\mathcal{H}(\underline{\underline{\alpha}})$ containing $X_0^{(0)}$
are the last $M(\underline{\underline{\alpha}}) - s_1 + 1$ lines. If
we remove them, we are left with a square matrix with colums:
\[
\Big(
\mathcal{C}_{\alpha^N_1}(X_N)
\ \Big\vert \
\mathcal{C}_{\alpha^N_2}(X_N)
\ \Big\vert \
\cdots
\ \Big\vert \
\mathcal{C}_{\alpha^N_{n_N}}(X_N)
\ \Big\vert \
\mathcal{C}_{\alpha^{N-1}_1}(X_{N-1})
\ \Big\vert \
\mathcal{C}_{\alpha^{N-1}_2}(X_{N-1})
\ \Big\vert \
\cdots
\ \Big\vert \
\mathcal{C}_{\alpha^1_{n_1}}(X_1)
\Big).
\]
If, in these columns, we replace $X_i$ by $x_{0i} = \frac{X_i}{X_0}$,
we recover the matrix $\Delta_0(\underline{\underline{\alpha}})$. We
shall generalize this construction for all homogeneous variables
$X_j$.

\begin{Definition}
Let $j$ be the index of some homogeneous variable $X_j$. We define the
matrix $\Delta_j(\underline{\underline{\alpha}})$ to be the matrix
$\mathcal{H}(\underline{\underline{\alpha}})$ in which we removed all
lines and columns containing the variable $X_j^{(0)}$, and where we
replaced every other homogeneous variable $X_i$ by the affine variable
$\frac{X_i}{X_j}$.
\end{Definition}

The determinants of the matrices $\Delta_j
(\underline{\underline{\alpha}})$ will correspond to the
determinants of the dehomogenizations of the matrix
$\mathcal{H}(\underline{\underline{\alpha}})$.

\begin{Theorem}
\label{main3}
For every integer $j \in \llbracket 0, N \rrbracket$:
\[
\det \Big(
\frac{1}{X_j}\mathcal{H}(\underline{\underline{\alpha}})
\Big)
=
\pm
\det \Delta_j(\underline{\underline{\alpha}}).
\]
\end{Theorem}

We deduce the following change of charts formula:
\[
\aligned
\det \Delta_0(\underline{\underline{\alpha}})
&=
\pm \det \Big(
\frac{1}{X_0}\mathcal{H}(\underline{\underline{\alpha}})
\Big) \\
&=
\pm \frac{X_j^{M(\underline{\underline{\alpha}}) + 1}}{X_0^{M(\underline{\underline{\alpha}}) + 1}}\det \Big(
\frac{1}{X_j}\mathcal{H}(\underline{\underline{\alpha}})
\Big) \\
&=
\pm \frac{1}{x_{jj}^{M(\underline{\underline{\alpha}}) + 1}}
\det \Delta_j(\underline{\underline{\alpha}}).
\endaligned
\]

Since in the other charts the highest pole has order
$M(\underline{\underline{\alpha}}) + 1$ at most, we deduce that the
local section $\Delta_0(\underline{\underline{\alpha}})$ extends
holomorphically to a global section of the twisted bundle
$E^{\text{GG}}_{k, \infty}(d)$ if $d \geqslant
M(\underline{\underline{\alpha}}) + 1$ and if $k$ is sufficiently
large.

\begin{Corollary}
Let $d$ be some fixed integer. Then, for every $k \geqslant d-1$ and
$\underline{\underline{\alpha}} \in \mathcal{B}^{N+}$ such that
$M(\underline{\underline{\alpha}}) \leqslant d-1$, we have:
\[
\Delta_0(\underline{\underline{\alpha}})
\in
\mathcal{H}^0(\mathbb{P}^N(\mathbb{C}), E^{\textbf{GG}}_{k, \infty}(d)).
\]
\end{Corollary}

We have thus found a free family of global polynomial sections.
Furthermore, the cardinality of this family is given by the

\begin{Lemma}
\label{cardinal_famille_libre}
The number of elements
$\underline{\underline{\alpha}} \in \mathcal{B}^{N+}$
 such that 
$M(\underline{\underline{\alpha}}) \leqslant d-1$
equals $(N+1)^d$
\end{Lemma}

\begin{proof}
We shall describe a bijection between the sets 
$\llbracket 0, N \rrbracket^d$
and 
$\mathcal{B}^{N+}$.
Let $(u_0, u_1, \dots, u_{d-1})$ be a $d$-tuple in
$\llbracket 0, N \rrbracket^d$.
Denote by $0 \leqslant \alpha_1^N < \alpha_2^N < \cdots < \alpha_{n_N}^N$ the sequence of indexes such that 
\[
u_{\alpha_i^N} = N
\quad \quad
\text{for all} \  i \in \llbracket 1, n_N \rrbracket.
\]
Now, in the $d$-tuple $(u_0, u_1, \dots, u_{d-1})$, remove the terms equal to $N$, and relabel the remaining terms in order to have a $(d-n_N)$-tuple $(u_{n_N}, u_{n_N + 1}, \dots, u_{d-1})$
in
$\llbracket 0, N-1 \rrbracket^{d-n_N}$.
Using this new tuple, we define the sequence 
$n_N \leqslant \alpha_1^{N-1} < \alpha_2^{N-1} < \cdots < \alpha_{n_{N-1}}^{N-1}$ of indexes such that
\[
u_{\alpha_i^{N-1}} = N-1
\quad \quad
\text{for all} \  i \in \llbracket 1, n_{N-1} \rrbracket.
\]
By repeating this process, we define $N$ stricly increasing sequences 
\[
\underline{\alpha}^N, 
\underline{\alpha}^{N-1},
\dots,
\underline{\alpha}^1,
\]
which satisfy
\[
s_{i + 1} \leqslant \alpha_i^j
\quad \quad 
\text{for all} \ j \in \llbracket 1, N \rrbracket, i \in \llbracket 1, n_{j} \rrbracket,
\]
so that $(\underline{\alpha}^1, \underline{\alpha}^2, \dots, \underline{\alpha}^N)$ is an element in $\mathcal{B}^{N+}$.

This construction is clearly reversible, so that it defines a bijection.
\end{proof}

We conclude that we have the following lower bound:

\begin{Corollary}
For every integer $d$ and every $k \geqslant d-1$, the dimension of
the space of global sections of the twisted bundle
$E^{\textbf{GG}}_{k, \infty}(d)$ is greater or equal to $(N+1)^d$:
\[
\dim \mathcal{H}^0(\mathbb{P}^N(\mathbb{C}), E^{\textbf{GG}}_{k, \infty}(d))
\geqslant
(N+1)^d.
\]
\end{Corollary}

We will now give the proof of Theorem~\ref{main3}.

\begin{proof}
In the matrix $\mathcal{H}(\underline{\alpha})$, 
there are some columns in the variable $X_i$:
\[
\mathcal{C}_{\beta_1}(X_i), 
\mathcal{C}_{\beta_2}(X_i), 
\dots, 
\mathcal{C}_{\beta_m}(X_i).
\]
The entry of $\mathcal{C}_{\beta_t}(X_i)$ on line $l$ is:
\begin{equation} \label{lignel}
\binom{\beta_t}{l-1}X_i^{(\beta_t - l + 1)},
\end{equation}
with the conventions:
\[
\aligned
\binom{a}{b}
&=
0
\quad \quad 
\text{if} \ a < b, \\
X_i^{(\centersmallbullet)}
&=
0
\quad \quad 
\text{if} \ \centersmallbullet < 0.
\endaligned
\]
Now, consider some column in the variable $X_j$ in
$\mathcal{H}(\underline{\alpha})$:
\[
\mathcal{C}_{\gamma}(X_j),
\]
which has, on the line $l$, the entry:
\[
\binom{\gamma}{l-1}X_j^{(\gamma - l + 1)}
=
\binom{\gamma}{l-1}
\sum_{q=0}^{\gamma - l + 1}
\binom{\gamma - l + 1}{q}
x_{ij}^{(q)}X_i^{(\gamma - l + 1 - q)}.
\]
In particular, there is the term :
\[
\binom{\gamma}{l-1}
\binom{\gamma - l + 1}{\gamma - \beta_t}
X_i^{(\beta_t - l + 1)}
x_{ij}^{(\gamma - \beta_t)}.
\]
The ratio of this term and~\eqref{lignel} takes the nice form:
\[
\frac{
\binom{\gamma}{l-1}
\binom{\gamma - l + 1}{\gamma - \beta_t}
X_i^{(\beta_i - l + 1)}
x_{ij}^{(\gamma - \beta_t)}
}{
\binom{\beta_t}{l-1}
X_i^{(\beta_t-l+1)}
}
=
\binom{\gamma}{\beta_t}x_{ij}^{(\gamma - \beta_t)}.
\]
Thus, by making the following operation on the columns:
\[
\mathcal{C}_{\gamma}(X_j) 
\longleftarrow 
\mathcal{C}_{\gamma}(X_j) 
- 
\sum_{t=1}^m
\binom{\gamma}{\beta_t}x_{ij}^{(\gamma - \beta_t)}
\mathcal{C}_{\beta_t}(X_i),
\]
we get rid of all the terms containing one of the following variables:
\[
x_{ij}^{(\gamma - \beta_1)}, 
x_{ij}^{(\gamma - \beta_2)}, 
\dots, 
x_{ij}^{(\gamma-\beta_m)}.
\]

We do the same operation on all the columns that are in the variable
$X_j$, with $j \neq i$.  On some column $\mathcal{C}_{\gamma}(X_j)$,
the entry on line $l$ is now:
\[
\binom{\gamma}{l-1}
\sum_{\underset{q \neq \gamma - \beta_t}{q=0}}^{\gamma - l + 1}
\binom{\gamma - l + 1}{q}x_{ij}^{(q)}X_i^{(\gamma - l + 1 - q)}.
\]
On some fixed line $l$ with $l \not\in \{\beta_1 + 1, \beta_2 + 1,
\dots, \beta_m + 1\}$, the term corresponding to $q = \gamma - l + 1$
(which is the last term of the sum), is:
\[
\binom{\gamma}{l-1}
x_{ij}^{(\gamma - l + 1)}
X_i.
\]
By using operations on the lines, we will make every other term
disappear.

We will proceed by induction, starting on the last line and going
upwards.

Let $l$ denote the number of the last line. By the construction of
$\mathcal{H}(\underline{\alpha})$, the entry on the $l^{\text{th}}$
line of $\mathcal{C}_{\gamma}(X_j)$ is either $0$ or $x_{ij}X_i$. So
we can start the induction.

Suppose that the induction hypothesis is true for all lines $l
\geqslant L+1$, where $L$ is the number of the line on which we want
to make the induction.

So far, the column looks like this:
\begin{center}
$\begin{pmatrix}
\sum_{q=0}^{\gamma} \binom{\gamma}{q}x_{ij}^{(q)}X_i^{(\gamma - q)} \\
\vdots \\
\binom{\gamma}{L-1}
\sum_{q=0}^{\gamma} \binom{\gamma-L+1}{q}x{ij}^{(q)}X_i^{(\gamma-L+1-q)} \\
\binom{\gamma}{L}x_{ij}^{(\gamma - L)}X_i \\
\binom{\gamma}{L+1}x_{ij}^{(\gamma-L-1)}X_i \\
\vdots
\end{pmatrix},$
\begin{tikzpicture}[overlay,remember picture]
\node[anchor=east] at ($(pic cs:lineOne)+(-6.6,.3)$) 
{\small $L$-th line $\rightarrow$};
\node[anchor=east] at ($(pic cs:lineOne)+(-6.6,-.7)$) 
{\small lines cleared by induction $\Bigg\{$};
\end{tikzpicture}
\end{center}
where we remove all the terms in 
$x_{ij}^{(\gamma - \beta_t)}$ for all $t=1, \dots, m$.

Now, for very $1 \leqslant k \leqslant \gamma - L + 1$ such that $k
\neq \beta_t - L + 1$ for all $t$, there is, on line number $L$, the
term:
\[
\binom{\gamma}{L-1}
\binom{\gamma - L + 1}{k}
x_{ij}^{(\gamma - L + 1 - k)}
X_i^{(k)}.
\]
On the other hand, the entry on line $L + K$, which has already been
cleared during the induction, is:
\[
\binom{\gamma}{L+k-1}
x_{ij}^{(\gamma - L + 1 - k)}
X_i.
\]
Therefore, the ratio of the two previous terms is equal to:
\[
\binom{L+k-1}{k}
\frac{X_i^{(k)}}{X_i}.
\]
Thus we may apply the following operation on the lines:
\[
\text{line}(L) 
\ \longleftarrow \ 
\text{line}(L)
-
\sum_{k=1}^{\gamma - L + 1}
\binom{L+k-1}{k}
\frac{X_i^{(k)}}{X_i}
\ \text{line}(L+k).
\]
On line $L$ there will be only $\binom{\gamma}{L-1}
x_{ij}^{(\gamma -L+1)}X_i$ left, completing the induction. 

Divide now every entry by $X_i$. Then our matrix looks like this:
\begin{center}
$\begin{pmatrix}
x_{ij}^{(\gamma_1)} & x_{ij}^{(\gamma_2)} & \cdots & \cdots & \bigast & \bigast & \bigast \\
\binom{\gamma_1}{1}x_{ij}^{(\gamma_1) - 1} &
\binom{\gamma_2}{1}x_{ij}^{(\gamma_2) - 1} &
\cdots & \cdots & \bigast & \bigast & \bigast \\
\vdots & \vdots & \vdots & \vdots & \bigast & \bigast & \bigast \\
0 & 0 & 0 & \cdots & 1 & \bigast & \bigast \\
\vdots & \vdots & \vdots & \vdots & 0 & \bigast & \bigast \\
0 & 0 & 0 & \cdots & 0 & 1 & \bigast \\
\vdots & \vdots & \vdots & \vdots & 0 & 0 & \bigast \\
0 & 0 & 0 & \cdots & 0 & 0 & 1 \\
\end{pmatrix}$.
\begin{tikzpicture}[overlay,remember picture]
\draw[color=red] (-2.1,2.1) -- (-2.1,-2.1);
\end{tikzpicture}
\end{center}
The red line separates the $X_j$ columns from the $X_i$ columns.

We see that, for the determinant not to be zero, one has to take all
the $1$'s in the right part of the matrix. This leaves us with the
matrix described by the theorem.
\end{proof}

%%%%%%%%%%%%%%%%%%%%%%%%%%%%%%%%%%%%%%%%%%%%%%%%%%%%%%%%%%%%%%%%%%%%%%
\subsection{The non-vanishing pole}
%%%%%%%%%%%%%%%%%%%%%%%%%%%%%%%%%%%%%%%%%%%%%%%%%%%%%%%%%%%%%%%%%%%%%%

In this section we will work in the chart $U_1$, 
and we will therefore use the lightened notations:
\[
y_i
=
x_{1i}
\ \ \text{for all} \ i \in \llbracket 1, N \rrbracket.
\]
We use the same notations for the monomials in the variables $\{
y_i^{(l)} \}$ as the ones we used for the monomials in $\{
x_{0i}^{(l)} \}$: if $\underline{\underline{\alpha}} =
(\underline{\alpha}^1, \underline{\alpha}^2, \dots,
\underline{\alpha}^N)$ is some element in $\mathcal{A}^N$, then we
denote:
\[
y^{(\underline{\underline{\alpha}})}
=
y_N^{(\alpha^N_1)}y_N^{(\alpha^N_2)} \cdots y_N^{(\alpha^N_{n_N})}
y_{N-1}^{(\alpha^{N-1}_1)}y_{N-1}^{(\alpha^{N-1}_2)} \cdots
y_1^{(\alpha^1_{n_1})}.
\]

Furthermore, we use the same order on those monomials as the one we
used for the monomials in $\{ x_{0i}^{(l)} \}$. We start by ordering
the variables:
\[
\aligned
1 &\prec y_N \prec y_N' \prec y_N'' \prec \cdots \\
&\prec y_{N-1} \prec y_{N-1}' \prec y_{N-1}'' \prec \cdots\\
&\,\,\:\vdots \\
&\prec y_1 \prec y_1' \prec y_1'' \prec \cdots .
\endaligned
\]

Now, let $\chi_1\chi_2\cdots \chi_n$ and $\chi'_1\chi'_2\cdots
\chi'_m$ be two monomials in the variables $\{ y_i^{(l)} \}$, and
suppose the variables are already ordered:
\[
\chi_1 \leqslant \chi_2 \leqslant \cdots \leqslant \chi_n
\ \ \ \text{and} \ \ \
\chi_1' \leqslant \chi_2' \leqslant \cdots \leqslant \chi_m'.
\]
We will say that $\chi_1 \chi_2 \cdots \chi_n \prec
\chi_1'\chi_2'\cdots \chi_{m}'$ if and only if:
\[
\chi_n \prec \chi_m',
\]
or
\[
\chi_n = \chi_m'
\ \ \text{and} \ 
\chi_1 \chi_2 \cdots \chi_{n-1} \prec \chi_1'\chi_2'\cdots \chi_{m-1}'.
\]

By induction, this defines a total ordering on the set of monomials in
the variables $\{ y_i^{(l)} \}$.

We know that $\det \Delta_0(\underline{\underline{\alpha}})$ is a
local section of the Green-Griffiths bundle over the open set $U_0$,
and that after the change of chart it becomes
$\frac{1}{y_1^{M(\underline{\underline{\alpha}}) + 1}}\det
\Delta_1(\underline{\underline{\alpha}})$ over the open set $U_1$. If
we expand the determinant, we get a linear combination of
rational terms having the form:
\[
\frac{y^{(\underline{\underline{a}})}}{y_1^s},
\]
where $1 \leqslant a^1_1$, so that the fraction is irreductible. We
will define an ordering 
on those rational terms that generalizes the ordering
we introduced on the monomials.

For two rational terms
$\frac{y^{(\underline{\underline{a}})}}{y^s}$ 
and 
$\frac{y^{(\underline{\underline{b}})}}{y^{s'}}$,
we declare that:
\[
\frac{y^{(\underline{\underline{a}})}}{y^s}
\prec
\frac{y^{(\underline{\underline{b}})}}{y^{s'}}
\]
if:
\[
s > s'
\]
or:
\[
s = s \ \ \ \text{and} \ \ \ y^{(\underline{\underline{a}})} \prec y^{(\underline{\underline{b}})}.
\]

We will now describe the smallest term that appears in the expansion
of $\frac{1}{y_1^{M(\underline{\underline{\alpha}}) + 1}}\det
\Delta_1(\underline{\underline{\alpha}})$.  We start by noticing that,
since $s_2 \leqslant \alpha^N_1 < \alpha^N_2 < \cdots <
\alpha^N_{n_N}$, the variables $X_1^{(0)}$ and $X_0^{(0)}$ will only
appear in $\mathcal{H}(\underline{\underline{\alpha}})$ in the bottom
right submatrix of size $(M(\underline{\underline{\alpha}}) - s_2 + 1)
\times (M(\underline{\underline{\alpha}}) - s_2 + 1)$:
\[
\mathcal{H}(\underline{\underline{\alpha}})
=
\begin{pmatrix}
X_N^{(\alpha^N_1)} & X_N^{(\alpha^N_2)} & \cdots & X_1^{(\alpha^1_1)} & \cdots X_0^{(M(\underline{\underline{\alpha}}))} \\
& & & & \\
& & & & \\
& & & & \\
& & & \vdots & \\
& & & X_1 & \\
& & & & \\
& & & & X_0 \\
\end{pmatrix}.
\begin{tikzpicture}[overlay,remember picture]
\draw[color=red] (-3.5,.1) -- (-3.5,-1.9);
\draw[color=red] (-3.5,.1) -- (-.6,.1);
\draw[color=red] (-.6,-1.9) -- (-3.5,-1.9);
\draw[color=red] (-.6,-1.9) -- (-.6,.1);
\draw[color=green] (-3.7,.2) -- (-6.8,.2);
\draw[color=green] (-3.7,.2) -- (-3.7,2.1);
\draw[color=green] (-6.8,2.1) -- (-6.8,.2);
\draw[color=green] (-6.8,2.1) -- (-3.7,2.1);
\end{tikzpicture}
\]

Thus, when removing the lines and columns containing $X_1^{(0)}$, the
green, upper left submatrix is not affected. The product of the
diagonal entries of this submatrix yields, up to some nonzero
multiplicative constant:
\[
X_N^{(\alpha^N_1)}X_N^{(\alpha^N_2 - 1)}X_N^{(\alpha^N_3 - 2)} \cdots
X_{N-1}^{(\alpha^{N-1}_1 - s_N)}X_{N-1}^{(\alpha^{N-1}_1 - s_N - 1)} \cdots .
\]
We now need to describe the diagonal of the bottom right submatrix. If
$\alpha^1_{n_1} = M(\underline{\underline{\alpha}})$, we can notice
that this submatrix is the same as $\mathcal{H}(\alpha^1_1 - s_2,
\alpha^1_2 - s_2, \dots, \alpha^1_{n_1} - s_2)$ with other binomial
coefficients. After dehomogenizing by
$\frac{1}{X_0^{M(\underline{\underline{\alpha}})}}$, it will
correspond to $\Delta_1(\underline{\alpha}^1)$. The product of the
diagonal entries of this matrix has already been computed in 
Section~\ref{non_vanishing_pole}, and is equal to:
\[
(y')^{\alpha^1_{n_1} - \alpha^1_{n_1 - 1} - 1}
(y'')^{\alpha^1_{n_1 - 1} - \alpha^1_{n_1 - 2} - 1}
\cdots
(y^{(n_1)})^{\alpha^1_1 - s_2}.
\]

\begin{Property}
If $\underline{\underline{\alpha}}$ is some $N$-tuple in
$\mathcal{A}^N$, such that $M(\underline{\underline{\alpha}}) =
\alpha^1_{n_1}$, then the smallest term of
$\frac{1}{y^{M(\underline{\underline{\alpha}}) +
1}_1}\Delta_1(\underline{\underline{\alpha}})$ is:
\[
\aligned
\frac{1}{y_1^{M(\underline{\underline{\alpha}}) + 1}}
\Big(
\prod_{j = 2}^{N}
\prod_{i = 1}^{n_j}
y_j^{(\alpha^j_i - i + 1 - s_{j-1})}
\Big)
&(y_1')^{\alpha^1_{n_1} - \alpha^1_{n_1 - 1} - 1}
(y_1'')^{\alpha^1_{n_1 - 1} - \alpha^1_{n_1 - 2} - 1}
\cdots \\
&(y_1^{(n_1-1)})^{\alpha^1_2 - \alpha^1_1 - 1}
(y_1^{(n_1)})^{\alpha^1_1 - s_2}.
\endaligned
\]
\end{Property}

Before giving the proof, we will make some comments on this term.

\begin{itemize}

\item The number of factors in the numerator is equal to $M(\underline{\underline{\alpha}}) - n_1 + 1$, which is the size of $\Delta_1(\underline{\underline{\alpha}})$.

\item For every $j \in \llbracket 2, N \rrbracket$, the variable $y_j$ and its derivatives $y_j', y_j'', \dots$ appear exactly $n_j$ times in the numerator.

\item The biggest derivative under which $y_1$ appears is $n_1$.
\end{itemize}

\begin{proof}
The matrix $\Delta_1(\underline{\underline{\alpha}})$ satisfies the
conditions of Theorem~\ref{minth}. Hence, its minimal monomial is
given by the product of the diagonal entries. Furthermore, this
product does not contain the factor $y_1$. Hence, the rational term we
obtain after dividing by $y_1^{M(\underline{\underline{\alpha}})}$ has
the biggest possible power on the denominator. This rational term is
therefore the smallest, and this very property guarantees us that it
does not vanish.
\end{proof}

Suppose now that $\alpha^1_{n_1} \neq
M(\underline{\underline{\alpha}})$, and consider the set
$J_M(\underline{\underline{\alpha}}) \subset \llbracket 1, N-1
\rrbracket$ of indices $j$ such that $\alpha^j_{n_j} =
M(\underline{\underline{\alpha}})$. By definition of $M$,
$J_M(\underline{\underline{\alpha}})$ is not empty. The last row of
the homogeneous matrix $\mathcal{H}(\underline{\underline{\alpha}})$
is:
\[
\begin{pmatrix}
0 & \cdots & X_{j_1} & \cdots & X_{j_2} & \cdots & X_{j_v} & \cdots & X_0
\end{pmatrix},
\]
where $J_M(\underline{\underline{\alpha}}) = \{j_1, j_2, \cdots,
j_v\}$. Since $X_1^{(0)}$ does not appear in this row (this is a
direct consequence of $\alpha^1_{n_1} \neq
M(\underline{\underline{\alpha}})$), the row is not removed in
$\Delta_1(\underline{\underline{\alpha}})$. Hence, the last row of
$\Delta_1(\underline{\underline{\alpha}})$ is:
\[
\begin{pmatrix}
0 & \cdots & y_{j_1} & \cdots & y_{j_2} & \cdots & y_{j_v} & \cdots & y_1
\end{pmatrix}.
\]

The diagonal term of $\det \Delta_1(\underline{\underline{\alpha}})$
therefore contains a $y_1$ factor, which makes it a bad candidate for
being the smallest term. We will thus have to consider some other
terms in the determinant.

For any integer $c \leqslant M(\underline{\underline{\alpha}}) + 1$,
we denote the matrix $\Delta_1(\underline{\underline{\alpha}})$ in
which we remove the last row and the $c^{\text{th}}$ column by
$\Delta_1(\underline{\underline{\alpha}})^{[c]}$. Developping the
determinant $\Delta_1$ along the last line, we get:
\[
\det \Delta_1(\underline{\underline{\alpha}})
=
\sum_{t \in J_M(\underline{\underline{\alpha}})}
(-1)^{M(\underline{\underline{\alpha}}) + s_t}y_t\det \Delta_1(\underline{\underline{\alpha}})^{[s_t]}
+
y_1\Delta_1(\underline{\underline{\alpha}})^{[s_t]}.
\]

For every $t \in J_M(\underline{\underline{\alpha}})$, the smallest
term in $\det \Delta_1(\underline{\underline{\alpha}})^{[s_t]}$ is the
diagonal term, for the matrix satisfies the condition of theorem
\eqref{minth}. We shall fix the integer $t$, and describe the diagonal
terms in $\Delta_1(\underline{\underline{\alpha}})^{[s_t]}$.

We can cut $\Delta_1(\underline{\underline{\alpha}})^{[s_t]}$ in three
vertical parts:

\begin{itemize}

\item The first part containing the columns going from $1$ to $s_t - 1$, in which the diagonal terms are unchanged, and they are:
\[
y_N^{(\alpha^N_1)}, \ y_N^{(\alpha^N_2 - 1)}, 
\ \dots, \ 
y_{t}^{(\alpha^t_{n_t - 1} - s_t + 2)}.
\]

\item The second part containing the columns going from $s_t$ to $s_2
- 1$, in which the diagonal entries are the entries that were above
the diagonal in $\Delta_1(\underline{\underline{\alpha}})$, because
we skipped a column. These terms are:
\[
y_{t-1}^{(\alpha^t_1 - s_t + 1)}, \
y_{t-1}^{(\alpha^t_2 - s_t + 0)}, \ 
\dots, \
y_{2}^{(\alpha^2_{n_2} - s_2 + 2)}.
\]

\item The third part containing the columns in $y_1$. In these
columns, we have removed the last line and all the lines numbered
$\alpha^N_1 + 1, \alpha^N_2 + 1, \dots, \alpha^N_{n_N} + 1$.

\end{itemize}

To find the diagonal entries for the columns in the variable $y_1$, we
will search them in $\mathcal{H}(\underline{\underline{\alpha}})$. In
this matrix, the lines that are strictly between the last line and
line number $\alpha^1_{n_1}+1$ will each yield a diagonal term that is
in $X_0'$. Then, the lines that are strictly between line
$\alpha^1_{n_1} + 1$ and $\alpha^1_{n_1 - 1} + 1$ will each yield a
diagonal term that is in $X_0''$.
\[
\ytableausetup{centertableaux, boxsize=1.2em}
\begin{pmatrix}
&  &  \cdots & \vdots & \vdots & \vdots & \vdots \\
&  &  \cdots & 
\begin{ytableau}
X_0''
\end{ytableau}
&  X_0''' & \vdots & \vdots  \\
 &  &  \cdots & X_0' & 
\begin{ytableau}
X_0''
\end{ytableau}
& X_0''' & \vdots  \\
&  &  \cdots & X_0 & X_0' & X_0'' & X_0''' \\
&  &  \cdots & 0 & X_0 & 
\begin{ytableau}
X_0'
\end{ytableau}
& X_0'' \\
&  &  \cdots & 0 & 0 & X_0 & 
\begin{ytableau}
X_0'
\end{ytableau} \\
&  &  \cdots & 0 & 0 & 0 & X_0 \\
\end{pmatrix}
\begin{tikzpicture}[overlay,remember picture]
\draw[color=red] (-5.1,-1.5) -- (-0.4,-1.5);
\draw[color=red] (-5.1,-.2) -- (-0.4,-.2);
\end{tikzpicture}
\]

This continues, and the lines strictly between lines number
$\alpha^1_2 + 1$ and $\alpha^1_1 + 1$ will yield diagonal terms in
$X_0^{(n_1)}$. Finally, there will be the entries that are above line
number $\alpha^1_1 + 1$. These lines will each give a factor in
$X_0^{(n_1 + 1)}$. To find out how many they are, we can use the fact
that there is a total number of $M(\underline{\underline{\alpha}}) + 1
- s_1$ columns in $X_0$, and that we have already found the diagonal
terms of:
\[
(M(\underline{\underline{\alpha}}) - \alpha^1_{n_1} - 1) + (\alpha^1_{n_1} - \alpha^1_{n_1 - 1} - 1) + \cdots + (\alpha^1_2 - \alpha^1_1 - 1)
=
M(\underline{\underline{\alpha}}) - \alpha^1_1 - n_1
\]
of them. The remaining number of factors in $X_0^{(n_1 + 1)}$ is thus
$\alpha^1_1 - s_2 + 1$.

We conclude that the smallest term in $y_t\det
\Delta_1(\underline{\underline{\alpha}})^{{[s_t]}}$ is:
\[
\aligned
y_t
\Big(
\prod_{j=t + 1}^{N}
\prod_{i = 1}^{n_j}
y_t^{(\alpha^j_i - i + 1 - s_{j+1})}
\Big)
&\Big(
\prod_{i = 1}^{n_t - 1}
y_t^{(\alpha^j_i - i + 1 - s_{t_v + 1})}
\Big)
\Big(
\prod_{j = 2}^{t - 1}
\prod_{i = 1}^{n_j}
y_j^{(\alpha^j_i - i + 2 - s_{j+1})}
\Big) \\
&(y_1')^{M(\underline{\underline{\alpha}}) - \alpha^1_{n_1} - 1}
(y_1'')^{\alpha^1_{n_1} - \alpha^1_{n_1 - 1}}
(y_1''')^{\alpha^1_{n_1-1} - \alpha^1_{n_1-2}}
\cdots \\
&(y_1^{(n_1)})^{\alpha^1_2 - \alpha^1_1 - 1}
(y_1^{(n_1 + 1)})^{\alpha^1_1 - s_2 + 1}.
\endaligned
\]

Consider now some other integer $u \in
J_M(\underline{\underline{\alpha}})$. We can suppose that $t > u$. We
will compare the smallest terms of $y_t\det
\Delta_1(\underline{\underline{\alpha}})^{{[s_t]}}$ and $y_u\det
\Delta_1(\underline{\underline{\alpha}})^{{[s_u]}}$. These two
smallest terms will share the factors in the variables $y_N, y_{N-1},
\dots, y_{t+1}$, and in $y_{u-1}, y_{u - 2}, \dots, y_1$. They will be
differences for the middle variables $y_t, y_{t-1}, \dots, y_u$. Since
$y_u$ is the biggest variable, it can (and will) determine by itself
which of the two terms is the smallest.

In \ $y_t\det \Delta_1(\underline{\underline{\alpha}})^{{[s_t]}}$, the
variables $y_u, y_u', \dots$ appear in a product:
\[
\prod_{i = 1}^{n_u}
y_u^{(\alpha^u_i - i + 2 - s_{u-1})}.
\]

In \ $y_u\det \Delta_1(\underline{\underline{\alpha}})^{{[s_u]}}$, the variables $y_u, y_u', \dots$ appear in a product:
\[
y_u
\Big(
\prod_{i = 1}^{n_u - 1}
y_u^{(\alpha^u_i - i + 1 - s_{u-1})}
\Big).
\]

The derivatives in the first product are greater than the ones in the
second product, so we conclude that if $t > u$, then the smallest term
of $y_u\det \Delta_1(\underline{\underline{\alpha}})^{{[s_u]}}$ is
smaller than the smallest term of $y_t\det
\Delta_1(\underline{\underline{\alpha}})^{{[s_t]}}$. We deduce the

\begin{Property}
Let $\underline{\underline{\alpha}}$ be a $N$-tuple of sequences in
$\mathcal{B}^{N+}$ such that $\alpha^1_{n_1} \neq
M(\underline{\underline{\alpha}})$, and let $t_v$ be the smallest
index in $J_M(\underline{\underline{\alpha}})$. The smallest rational
term in $\frac{\det \Delta_1( \underline{ \underline{\alpha}})}{
y_1^{M( \underline{ \underline{\alpha}})}}$ is:
\[
\aligned
\frac{y_{t_v}}{y_1^{M(\underline{\underline{\alpha}}) + 1}}
&\Big(
\prod_{j=t_v + 1}^{N}
\prod_{i = 1}^{n_j}
y_t^{(\alpha^j_i - i + 1 - s_{j-1})}
\Big)
\Big(
\prod_{i = 1}^{n_{t_v} - 1}
y_{t_v}^{(\alpha^{t_v}_i - i + 1 - s_{t_v-1})}
\Big)
\Big(
\prod_{j = 2}^{t_v - 1}
\prod_{i = 1}^{n_j}
y_j^{(\alpha^j_i - i + 2 - s_{j-1})}
\Big) \\
&(y_1')^{M(\underline{\underline{\alpha}}) - \alpha^1_{n_1} - 1}
(y_1'')^{\alpha^1_{n_1} - \alpha^1_{n_1 - 1}}
(y_1''')^{\alpha^1_{n_1-1} - \alpha^1_{n_1-2}}
\cdots \\
&(y_1^{(n_1)})^{\alpha^1_2 - \alpha^1_1 - 1}
(y_1^{(n_1 + 1)})^{\alpha^1_1 - s_2 + 1}.
\endaligned
\]
\end{Property}

There are four important comments to make here:

\begin{itemize}

\item The number of factors in the numerator is equal to
$M(\underline{\underline{\alpha}}) - n_1 + 1$, which is the size of
$\Delta_1(\underline{\underline{\alpha}})$.

\item For every $j \in \llbracket 2, N \rrbracket$, the variable $y_j$
and its derivatives $y_j', y_j'', \dots$ appear exactly $n_j$ times
in the numerator.

\item The variable $y_1^{(n_1 + 1)}$ appears as a factor, as
$\alpha^1_1 - s_2 + 1 \geqslant 1$.

\item Amongst the non-derivated variables $y_N \prec y_{N-1} \prec
\cdots \prec y_2$, the biggest one appearing as a factor in the
numerator is $y_{t_v}$.

\end{itemize}

Now that we have an explicit description of the minimal terms of the
$\frac{\det \Delta_1( \underline{ \underline{\alpha}})}{ y_1^{M(
\underline{ \underline{\alpha}})}}$, we can state our main lemma
that will be used to prove the final main
Theorem~{\ref{main-thm-end}}.

\begin{Lemma}
If $\underline{\underline{\alpha}} \neq \underline{\underline{\beta}}$
are two elements in $\mathcal{B}^{N+}$, then the minimal terms of:
\[
\frac{\det \Delta_1(\underline{\underline{\alpha}})}{y_1^{M(\underline{\underline{\alpha}}) + 1}}
\ \ \ \text{and} \ \ \
\frac{\det \Delta_1(\underline{\underline{\beta}})}{y_1^{M(\underline{\underline{\beta}}) + 1}}
\]
are different.
\end{Lemma}

\begin{proof}
Let $\underline{\underline{\alpha}} = (\underline{\alpha}^1,
\underline{\alpha}^2, \dots, \underline{\alpha}^N)$ and
$\underline{\underline{\beta}} = (\underline{\beta}^1,
\underline{\beta}^2, \dots, \underline{\beta}^N)$. We will denote the
elements of the sequences by $\underline{\alpha}^i = (\alpha^i_1,
\alpha^i_2, \dots, \alpha^i_{n_i})$ and $\underline{\beta}^i =
(\beta^i_1, \beta^i_2, \dots, \beta^i_{m_i})$ for all $i \in
\llbracket 2, N \rrbracket$.  We will suppose that $\frac{\det
\Delta_1(\underline{\underline{\alpha}})}{y_1^{M( \underline{ 
\underline{\alpha}})
+ 1}}$ and $\frac{\det
\Delta_1(\underline{\underline{\beta}})}{y_1^{M(\underline
{\underline{\beta}})
+ 1}}$ share the same smallest term. In that case,
$\underline{\underline{\alpha}}$ and $\underline{\underline{\beta}}$
must have the same maximum terms $M =
M(\underline{\underline{\alpha}}) = M(\underline{\underline{\beta}})$.
Furthermore, the number of factors constituting the numerator of the
smallest terms are respectively:
\[
M - n_1 + 1
\ \ \ \text{and} \ \ \
M - m_1 + 1.
\]
Thus we must have $n_1 = m_1$. We put $B = n_1 = m_1$. For all $i \in
\llbracket 2, N \rrbracket$, the number of factors in $y_i$ and its
derivatives $y_i', y_i'', \dots$ that are in the smallest terms must
also coincide, so that we have $n_i = m_i$ for all $i \in \llbracket
2, N \rrbracket$.

Suppose that $\alpha^1_{B} = M$ and that $\beta^1_{B} \neq M$. In that
case, the factor $y^{(B + 1)}$ appears in the smallest term of
$\frac{\det \Delta_1( \underline{ \underline{ \alpha}})}{ y_1^{M(
\underline{ \underline{\alpha}}) + 1}}$. However, it does not
appear in the smallest term of $\frac{\det \Delta_1( \underline{
\underline{\beta}})}{ y_1^{M( \underline{ \underline{\beta}}) +
1}}$, so that the two terms can not be equal. Hence we are in one
of two cases: The case where $\alpha^1_B$ and $\beta^1_B$ are both
equal to the maximum $M$, and the case where they both are not.

Suppose that $\alpha^1_B = \beta^1_B = M$. The smallest terms will
contain some factors in the variables $y_N, y_{N - 1}, \dots, y_2$ and
their derivatives. These factors are:
\[
\prod_{j = 2}^{N}
\prod_{i = 1}^{n_j}
y_j^{(\alpha^j_i - i + 1 - s_{j-1})}
\ \ \ \text{and} \ \ \
\prod_{j = 2}^{N}
\prod_{i = 1}^{m_j}
y_j^{(\beta^j_i - i + 1 - s_{j-1})}.
\]
For them to be equal, one must have $\alpha^j_i = \beta^j_i$ for all
$j \in \llbracket 2, N \rrbracket$ and $i \in \llbracket 1, n_j
\rrbracket$.

Finally, the two smallest terms will contain some factors in $y_1,
y_1', y_1'', \dots$, that are:
\[
(y_1')^{\alpha^1_B - \alpha^1_{B - 1} - 1}
(y_1'')^{\alpha^1_{B - 1} - \alpha^1_{B - 2} - 1}
\cdots
(y_1^{(B-1)})^{\alpha^1_2 - \alpha^1_1 - 1}
(y_1^{(B)})^{\alpha^1_1 - s_2}.
\]
and
\[
(y_1')^{\beta^1_B - \beta^1_{B - 1} - 1}
(y_1'')^{\beta^1_{B - 1} - \beta^1_{B - 2} - 1}
\cdots
(y_1^{(B-1)})^{\beta^1_2 - \beta^1_1 - 1}
(y_1^{(B)})^{\beta^1_1 - s_2}.
\]
We already know that $\alpha^1_B = \beta^1_B$. By induction, we can
then show that $\alpha^1_{B-t} = \beta^1_{B-t}$ for all $t \in
\llbracket 0, B-1 \rrbracket$, which proves that the elements
$\underline{\underline{\alpha}}$ and $\underline{\underline{\beta}}$
are equal.

Suppose now that $\alpha^1_B \neq M$ and that $\beta^1_B \neq M$. Let
$t(\underline{\underline{\alpha}})$ be the smallest index such that
$\alpha^t_{n_t} = M$, and let $t(\underline{\underline{\beta}})$ be
the smallest index such that $\beta^t_{m_t} = M$.  Amongst the
non-derivated variables $y_N < y_{N-1} < y_{N-2} < \cdots < y_2$, the
biggest one appearing in the smallest monomial of $\frac{\det
\Delta_1( \underline{ \underline{ \alpha}})}{ y_1^{M( \underline{
\underline{\alpha}}) + 1}}$ is
$y_{t({\underline{\underline{\alpha}}})}$, and the biggest one
appearing in the smallest monomial of $\frac{\det \Delta_1(
\underline{ \underline{\beta}})}{ y_1^{M( \underline{
\underline{\beta}}) + 1}}$ is $y_{t(\underline{ \underline{
\beta}})}$. Since the two smallest monomials are equal, they
must share their biggest variable in $y_N, y_{N-1}, \dots, y_2$. So we
have $t(\underline{\underline{\alpha}}) =
t(\underline{\underline{\beta}})$, and we denote this common index by
$t$.

If we compare the factors in $y_N, y_{N-1}, \dots, y_2$ and their
derivatives, we can conclude that $\alpha^j_{i} = \beta^j_i$ for all
$j \in \llbracket 2, N \rrbracket$ and $i \in \llbracket 1, n_j
\rrbracket$.  If we compare the factors in $y_1, y_1', y_1'', \dots$,
we see that we must have:
\[
\aligned
&(y_1')^{M - \alpha^1_{n_1} - 1}
(y_1'')^{\alpha^1_{n_1} - \alpha^1_{n_1 - 1}}
(y_1''')^{\alpha^1_{n_1-1} - \alpha^1_{n_1-2}}
\cdots 
(y_1^{(n_1)})^{\alpha^1_2 - \alpha^1_1 - 1}
(y_1^{(n_1 + 1)})^{\alpha^1_1 - s_2 + 1} \\
&=
(y_1')^{M- \beta^1_{n_1} - 1}
(y_1'')^{\beta^1_{n_1} - \beta^1_{n_1 - 1}}
(y_1''')^{\beta^1_{n_1-1} - \beta^1_{n_1-2}}
\cdots 
(y_1^{(n_1)})^{\beta^1_2 - \beta^1_1 - 1}
(y_1^{(n_1 + 1)})^{\beta^1_1 - s_2 + 1},
\endaligned
\]
which, by induction, implies that $\alpha^1_t = \alpha^1_t$ for all $t
\in \llbracket 1, n_1 \rrbracket$. We conclude that
$\underline{\underline{\alpha}} = \underline{\underline{\beta}}$.
\end{proof}

Since each local section $\det
\Delta_0(\underline{\underline{\alpha}})$ over $U_0$ will yield a
different smallest term when read in the chart $U_1$, we can conclude that
in every linear combination, there will be some term that will not
vanish, as it will be smallest of the smallest. In conclusion, we have
the following:

\begin{Theorem}
\label{main-thm-end}
For every $d \geqslant 1$ and $k \geqslant d-1$, the family:
\[
\Big\{
\Delta_0(\underline{\underline{\alpha}})
\Big\}_{M(\underline{\underline{\alpha}}) \leqslant d-1}
\]
is a basis of $H^0 \big( \mathbb{P}^N(\mathbb{C} \big),
E^{\text{GG}}_{k, n}(d))$.\qed
\end{Theorem}

%%%%%%%%%%%%%%%%%%%%%%%%%%%%%%%%%%%%%%%%%%%%%%%%%%%%%%%%%%%%%%%%%%%%%%
%%%%%%%%%%%%%%%%%%%%%%%%%%%%%%%%%%%%%%%%%%%%%%%%%%%%%%%%%%%%%%%%%%%%%%
%%%%%%%%%%%%%%%%%%%%%%%%%%%%%%%%%%%%%%%%%%%%%%%%%%%%%%%%%%%%%%%%%%%%%%

%%%%%%%%%%%%%%%%%%%%%%%%%%%%%%%%%%%%%%%%%%%%%%%%%%%%%%%%%%%%%%%%%%%%%%

\vfill

\noindent
Victor {\sc Chen}, 
\'Ecole Normale Supérieure Paris-Saclay, 
4 Avenue des Sciences, 91190 Gif-sur-Yvette, France.
{\bf victor.chen@ens-paris-saclay.fr} and
{\bf victor.chen01@hotmail.com}

\medskip

\noindent
Jo\"el {\sc Merker}, 
Laboratoire de Math\'ematiques d'Orsay,
CNRS, Universit\'e Paris-Saclay, 91405 Orsay Cedex, France. 
ORCID 0000-0003-2653-2147.
{\bf joel.merker@universite-paris-saclay.fr

%%%%%%%%%%%%%%%%%%%%%%%%%%%%%%%%%%%%%%%%%%%%%%%%%%%%%%%%%%%%%%%%%%%%%%
%%%%%%%%%%%%%%%%%%%%%%%%%%%%%%%%%%%%%%%%%%%%%%%%%%%%%%%%%%%%%%%%%%%%%%
%%%%%%%%%%%%%%%%%%%%%%%%%%%%%%%%%%%%%%%%%%%%%%%%%%%%%%%%%%%%%%%%%%%%%%
\end{document}